\documentclass[11pt]{article}
\usepackage{latexsym,amssymb,amsmath,amsthm,graphicx,float}

\newcommand{\C}{\mathbb{C}}

\newcommand{\<}{\langle}
\renewcommand{\>}{\rangle}

\newcommand{\eps}{\varepsilon}

\newcommand{\tr}{\mbox{tr}}

\newcommand{\conj}{\overline}
\newcommand{\beq}{\begin{equation}}
\newcommand{\eeq}{\end{equation}}
\renewcommand{\tilde}{\widetilde}
\renewcommand{\hat}{\widehat}
\newcommand{\bit}{\begin{itemize}}
\newcommand{\eit}{\end{itemize}}

\newtheorem{theorem}{Theorem}
\newtheorem{lemma}{Lemma}

\newtheorem{corollary}[theorem]{Corollary}

\title{Convex recovery from \\ interferometric measurements}
\author{Laurent Demanet and Vincent Jugnon \\ $\,$ \\ Department of Mathematics, MIT }
\date{July 2013, revised September 2016}		% remove for today's date

\begin{document}
\maketitle

%\begin{center}
%{\em Dedicated to George Papanicolaou on the occasion of his 70th birthday.}
%\end{center}

\begin{abstract}

This paper discusses some questions that arise when a linear inverse problem involving $Ax = b$ is reformulated in the interferometric framework, where quadratic combinations of $b$ are considered as data in place of $b$.

First, we show a deterministic recovery result for vectors $x$ from measurements of the form $(Ax)_i \conj{(Ax)_j}$ for some left-invertible $A$. Recovery is exact, or stable in the noisy case, when the couples $(i,j)$ are chosen as edges of a well-connected graph. One possible way of obtaining the solution is as a feasible point of a simple semidefinite program. Furthermore, we show how the proportionality constant in the error estimate depends on the spectral gap of a data-weighted graph Laplacian. 

Second, we present a new application of this formulation to interferometric waveform inversion, where products of the form $(Ax)_i \conj{(Ax)_j}$ in frequency encode generalized cross-correlations in time. We present numerical evidence that interferometric \emph{inversion} does not suffer from the loss of resolution generally associated with interferometric imaging, and can provide added robustness with respect to specific kinds of kinematic uncertainties in the forward model $A$.

\end{abstract}

{\bf Acknowledgments.} The authors would like to thank Amit Singer, George Papanicolaou, and Liliana Borcea for interesting discussions.  Some of the results in this paper were reported in the conference proceedings of the 2013 SEG annual meeting \cite{Jugnon-SEG}. This work was supported by  the Earth Resources Laboratory at MIT, AFOSR grants FA9550-12-1-0328 and FA9550-15-1-0078, ONR grant N00014-16-1-2122, NSF grant DMS-1255203, the Alfred P. Sloan Foundation, and Total SA.
\section{Introduction}

Throughout this paper, we consider complex quadratic measurements of $x \in \C^n$ of the form
\begin{equation}\label{eq:interf}
B_{ij} = (Ax)_i \conj{(Ax)_j}, \qquad (i,j) \in E,
\end{equation}
for certain well-chosen couples of indices $(i,j)$, a scenario that we qualify as ``interferometric". This combination is special in that it is symmetric in $x$, and of rank 1 with respect to the indices $i$ and $j$.

%Recovery of $x$ will be sought up to a global phase factor.

The regime that interests us is when the number $m$ of measurements, i.e., of couples $(i,j)$ in $E$, is comparable to the number $n$ of unknowns. While phaseless measurements $b_{i} = | (Ax)_i |^2$ only admit recovery when $A$ has very special structure -- such as, being a tall random matrix with Gaussian i.i.d. entries \cite{CandesEldar, DemanetHand} -- products $(Ax)_i \conj{(Ax)_j}$ for $i \ne j$ correspond to the idea of phase differences, hence encode much more information. As a consequence, stable recovery occurs under very general conditions: left-invertibility of $A$ and ``connectedness" of set $E$ of couples $(i,j)$. These conditions suffices to allow for $m$ to be on the order of $n$. Various algorithms return $x$ accurately up to a global phase; we mosty discuss variants of lifting with semidefinite relaxation in this paper. In contrast to other recovery results in matrix completion \cite{CandesRecht, RechtFazelParrilo}, \emph{no randomness is needed in the data model}, and our proof technique involves elementary spectral graph theory rather than dual certification or uncertainty principles.

The mathematical content of this paper is the formulation of a quadratic analogue of the well-known relative error bound
\begin{equation}\label{eq:LS}
\frac{\| x - x_0 \|}{\| x_0 \|} \leq \kappa(A) \, \frac{\| e \|}{\| b \|}
\end{equation}
for the least-squares solution of the overdetermined linear system $Ax=b$ with $b = Ax_0 + e$, and where $\kappa(A)$ is the condition number of $A$. Our result is that an inequality of the form (\ref{eq:LS}) still holds in the quadratic case, but with the \emph{square root of the spectral gap of a data-weighted graph Laplacian} in place of $\| b \|$ in the right-hand side. This spectral gap quantifies the connectedness of $E$, and has the proper homogeneity with respect to $b$.

The numerical results mostly concern the case when $A$ is a forward model that involves solving a linearized fixed-frequency wave equation, as in seismic or radar imaging. In case $x$ is a reflectivity, $Ax$ is the wavefield that results from an incoming wave being scattered by $x$, and $(Ax)_i \conj{(Ax)_j}$ has a meaning in the context of interferometry, as we explain in the next section. Our claims are that
\begin{itemize}
\item Stable recovery holds, and the choice of measurement set $E$ for which this is the case matches the theory in this paper;
\item Interferometric inversion yields no apparent loss of resolution when compared against the classical imaging methods; and
\item Some degree of robustness to specific model inaccuracies is observed in practice, although it is not explained by the theory in this paper.
\end{itemize}

% In that case $x$ is a reflectivity, $(A x)_i$ is a sampled wavefield in the frequency domain, more appropriately written as $u(\mathbf{x}_i, \omega_i)$ where $i$ indexes both the sensor location and the frequency. The significance of a combinatio

\subsection{Physical context: interferometry}

In optical imaging, an interference fringe of two (possibly complex-valued) wavefields $f(t)$ and $g(t)$, where $t$ is either a time or a space variable, is any combination of the form $|f(t) + g(t+t')|^2$. The sum is a result of the linearity of amplitudes in the fundamental equations of physics (such as Maxwell or Schr\"{o}dinger), while the modulus squared is simply the result of a detector measuring intensities. The cross term $2 \Re (f(t) \overline{g(t+t')})$ in the expansion of the squared modulus manifestly carries the information of destructive vs. constructive interference, hence is a continuous version of what we referred to earlier as an ``interferometric measurement".

In particular, when the two signals are sinusoidal at the same frequency, the interferometric combination highlights a phase difference. In astronomical interferometry, the delay $t'$ is for instance chosen so that the two signals interfere constructively, yielding better resolution. Interferometric synthetic aperture radar (InSAR) is a remote sensing technique that uses the fringe from two datasets taken at different times to infer small displacements. In X-ray ptychograhy \cite{Rodenburg}, imaging is done by undoing the interferometric combinations that the diffracted X-rays undergo from encoding masks. These are but three examples in a long list of applications.

Interferometry is also playing an increasingly important role in geophysical imaging, i.e., inversion of the elastic parameters of the portions of the Earth's upper crust from scattered seismic waves. In this context however, the signals are more often impulsive than monochromatic, and interferometry is done as part of the \emph{computational processing rather than the physical measurements}\footnote{Time reversal is an important exception not considered in this paper, where interferometry stems from experimental acquisition rather than processing.}. An interesting combination of two seismogram traces $f$ and $g$ at nearby receivers is then $\hat{f}(\omega) \overline{\hat{g}(\omega)}$, i.e., the Fourier transform of their cross-correlation. It highlights a time lag in the case when $f$ and $g$ are impulses. More generally, it will be important to also consider the cross-ambiguities $\hat{f}(\omega) \overline{\hat{g}(\omega')}$ where $\omega' \simeq \omega$.

Cross-correlations have been shown to play an important role in geophysical imaging, mostly because of their stability to statistical fluctuations of a scattering medium \cite{Blomgren} or an incoherent source \cite{Garnier, Lobkis}. Though seismic interferometry is a vast research area, both in the exploration and global contexts \cite{Borcea,  Fichtner, Hanasoge, Schuster-book, SchusterYu, Tromp, WapenaarFokkema}, explicit inversion of reflectivity parameters from interferometric data has to our knowledge only been considered in \cite{Dussaud-thesis, Jugnon-SEG}. Interferometric inversion offers great promise for model-robust imaging, i.e., recovery of reflectivity maps in a less-sensitive way on specific kinds of errors in the forward model.

Finally, interferometric measurements also play an important role in quantum optical imaging. See \cite{Schotland} for a nice solution to the inverse problem of recovering a scattering dielectric susceptibility from measurements of two-point correlation functions relative to two-photon entangled states.

As we revise this paper, we also note the recent success of interferometric inversion for passive synthetic-aperture radar (SAR) imaging, under the name low-rank matrix recovery \cite{Yazici}.

\subsection{Mathematical context and related work}

The setting of this paper is discrete, hence we let $i$ and $j$ in place of either a time of frequency variable. We also specialize to $f=g$, and we let $f = Ax$ to possibly allow an explanation of the signal $f$ by a linear forward model\footnote{Such as scattering from a reflectivity profile $x$ in the Born approximation, for which $A$ is a wave equation Green's function. Section \ref{sec:num} covers the description of $A$ in this context.} $A$.

The link between products of the form $f \, \overline{g}$ and squared measurements $|f + g|^2$ goes both ways, as shown by the polarization identity
\[
f_i \, \conj{f_j} = \frac{1}{4} \sum_{k=1}^4 e^{- i \pi k/2} |f_i + e^{i \pi k/2} f_j |^2.
\]
Hence any result of robust recovery of $f$, or $x$, from couples $f_i \, \conj{f_j}$, implies the same result for recovery from phaseless measurements of the form $|f_i + e^{i \pi k/2} f_j |^2$. This latter setting was precisely considerered by Cand\`{e}s et al. in \cite{CandesEldar}, where an application to X-ray diffraction imaging with a specific choice of masks is discussed. In \cite{Alexeev}, Alexeev et al. use the same polarization identity to design good measurements for phase retrieval, such that recovery is possible with $m = O(n)$.

Recovery of $f_i$ from $f_i \conj{f_j}$ for some $(i,j)$ when $| f_i | = 1$ (interferometric phase retrieval) can be seen a special case of the problem of angular synchronization considered by Singer \cite{Singer}. There, rotation matrices $R_i$ are to be recovered (up to a global rotation) from measurements of relative rotations $R_i R_j^{-1}$ for some $(i,j)$. This problem has an important application to cryo-electron microscopy, where the measurements of relative rotations are further corrupted in an a priori unknown fashion (i.e., the set $E$ is to be recovered as well). An impressive recovery result under a Bernoulli model of gross corruption, with a characterization of the critical probability, were recently obtained by Wang and Singer \cite{Wang}. The spectrum of an adapted graph Laplacian plays an important role in their analysis \cite{Bandeira}, much as it does in this paper. Singer and Cucuringu also considered the angular synchronization problem from the viewpoint of rigidity theory \cite{SingerCucuringu}. For the similar problem of recovery of positions from relative distance, with applications to sensor network localization, see for instance \cite{Javanmard}.

The algorithmic approach considered in this paper for solving interferometric inversion problems is to formulate them via lifting and semidefinite relaxation. This idea was considered by many groups in recent years \cite{CandesEldar, Chai,  Javanmard, Singer, Waldspurger}, and finds its origin in theoretical computer science \cite{Goemans}.

%TODO:
% importance of rank-1 measurements, what happens in the more general case - future work?
% link to work of Bandeira: 1/lambda_2 scaling and maybe other things
% role of expanders in Singer's work
% what if the diagonal is not present, take a log, nonsymmetric case - future work?
% mention link to rigidity: look up Singer's paper

%Polarization identity:
%\[
%\conj{x_i} x_j = \frac{1}{4} \sum_k i^k |x_i + i^k x_j |^2,
%\]
%and conversely,
%\[
%|x_i + i^k x_j |^2 = |x_i|^2 + |x_j|^2 - 2 \Re ( i^k x_i \conj{x_j}).
%\]

\section{Mathematical results}

\subsection{Recovery of unknown phases}

Let us start by describing the simpler problem of interferometric phase recovery, when $A = I$ and we furthermore assume $| x_i | = 1$. Given a vector $x_0 \in \C^n$ such that $|(x_0)_i| = 1$, a set $E$ of pairs $(i,j)$, and noisy interferometric data $B_{ij} = (x_0)_i \conj{(x_0)_j} + \eps_{ij}$, find a vector $x$ such that 
\beq\label{eq:P1}
|x_i| = 1, \qquad \sum_{(i,j) \in E} | x_i \conj{x_j} - B_{ij} |\leq \sigma,
\eeq
for some $\sigma > 0$. Here and below, if no heuristic is provided for $\sigma$, we may cast the problem as a minimization problem for the misfit and obtain $\sigma$ a posteriori.

The choice of the elementwise $\ell_1$ norm over $E$ is arbitrary, but convenient for the analysis in the sequel\footnote{The choice of $\ell_1$ norm as ``least unsquared deviation" is central in \cite{Wang} for the type of outlier-robust recovery behavior documented there.}. We aim to find situations in which this problem has a solution $x$ close to $x_0$, up to a global phase. Notice that $x_0$ is feasible for (\ref{eq:P1}), hence a solution exists, as soon as $\sigma \geq \sum_{{i,j} \in E} | \eps_{ij}|$.

The relaxation by \emph{lifting} of this problem is to find $X$ (a proxy for $xx^*$) such that 
\begin{align}\label{eq:P2}
& X_{ii} = 1, \qquad \sum_{(i,j) \in E} |X_{ij} - B_{ij}| \leq \sigma, \qquad X \succeq 0, \notag \\
& \mbox{then let $x$ be the top eigenvector of $X$ with $\| x \|^2 = n$.}
\end{align}
The notation $X \succeq 0$ means that $X$ is symmetric and positive semi-definite. Again, the feasibility problem (\ref{eq:P2}) has at least one solution ($X_0 = x_0 x_0^*$) as soon as $\sigma \geq \sum_{{i,j} \in E} | \eps_{ij}|$. 

%Our result of robust recovery uses basic notions of spectral graph theory that we now review. 

The set $E$ generates edges of a graph $G=(V,E)$, where the nodes in $V$ are indexed by $i$. Without loss of generality, we consider $E$ to be symmetric. By convention, $G$ does not contain loops, i.e., the diagonal $j=i$ is not part of $E$. (Measurements on the diagonal are not informative for the phase recovery problem, since $| (x_0)_i|^2 = 1$.)

The graph Laplacian on $G$ is
\[ L_{ij} = \left\{ \begin{array}{ll}
         d_i & \mbox{if $i = j$};\\
        -1 & \mbox{if $(i,j) \in E$};\\
        0 & \mbox{otherwise},\end{array} \right. 
\]
where $d_i$ is the node degree $d_i = \sum_{j: (i,j) \in E} 1$. Observe that $L$ is symmetric and $L \succeq 0$ by Gershgorin. Denote by $\lambda_1 \leq \lambda_2 \leq \ldots \leq \lambda_n$ the eigenvalues of $L$ sorted in increasing order. Then $\lambda_1 = 0$ with the constant eigenvector $v_1 = \frac{1}{\sqrt{n}}$. The second eigenvalue is zero if and only if $G$ has two or more disconnected components. When $\lambda_2 > 0$, its value is a measure of connectedness of the graph. Note that $\lambda_n \leq 2 d$ by Gershgorin again, where $d = \max_i d_i$ is the maximum degree.

%TODO [Relate $\lambda_2$ to Cheeger constant, expansion properties of the graph, mixing of random walks, degree of separation.]

Since $\lambda_1 = 0$, the second eigenvalue $\lambda_2$ is called the spectral gap. It is a central quantity in the study of expander graphs: it relates to
\bit
\item the edge expansion (Cheeger constant, large if $\lambda_2$ is large);
\item  the degree of separation between any two nodes (small if $\lambda_2$ is large); and 
\item the speed of mixing of a random walk on the graph (fast if $\lambda_2$ is large). 
\eit
More information about spectral graph theory can be found, e.g.,  in the lecture notes by Lovasz \cite{Lovasz}. It is easy to show with interlacing theorems that adding an edge to $E$, or removing a node from $V$, both increase $\lambda_2$. The spectral gap plays an important role in the following stability result.

% Courant-Weyl 

In the sequel, we denote the componentwise $\ell_1$ norm on the set $E$ by $\| \cdot \|_1$.

\begin{theorem}\label{teo:main1}
Assume $\| \eps \|_1 + \sigma \leq n \lambda_2 $, where $\lambda_2$ is the second eigenvalue of the graph Laplacian $L$ on $G$. Any solution $x$ of (\ref{eq:P1}) or (\ref{eq:P2}) obeys
\[
\| x - e^{i\alpha} x_0 \| \leq 4 \; \sqrt{\frac{\| \eps \|_1 + \sigma}{\lambda_2}},
\]
for some $\alpha \in [0,2 \pi)$.
\end{theorem}

The proof is in the appendix. Manifestly, recovery is exact (up to the global phase ambiguity encoded by the parameter $\alpha$, because the algorithm could return another vector multiplied by some $e^{i\alpha}$ over which there is no control.) as soon as $\eps = 0$ and $\sigma = 0$, provided $\lambda_2 \ne 0$, i.e., the graph $G$ is connected. The easiest way to construct expander graphs (graphs with large $\lambda_2$) is to set up a probabilistic model with a Bernoulli distribution for each edge in an i.i.d. fashion, a model known as the Erd\H{o}s-R\'{e}nyi random graph. It can be shown that such graphs have a spectral gap bounded away from zero, independently of $n$ and with high probability, with $m = O(n \log n)$ edges.

A stronger result is available when the noise $\eps$ originates at the level of $x_0$, i.e.,  $B = x_0 x_0^* + \eps$ has the form $(x_0 + e) (x_0 + e)^*$.

\begin{corollary}\label{teo:main1cor}
Assume $\eps = (x_0 + e) (x_0 + e)^* - x_0 x_0^*$ and $\sigma \leq n \lambda_2$, where $\lambda_2$ is the second eigenvalue of the graph Laplacian $L$ on $G$. Any solution $x$ of (\ref{eq:P1}) or (\ref{eq:P2}) obeys
\[
\| x - e^{i\alpha} x_0 \| \leq 4 \; \sqrt{\frac{\sigma}{\lambda_2}} + \| e \|,
\]
for some $\alpha \in [0,2 \pi)$.
\end{corollary}
\begin{proof} Apply theorem \ref{teo:main1} with $\eps = 0$,  $x_0 + e$ in place of $x_0$, then use the triangle inequality.
\end{proof}

In the setting of the corollary, problem (\ref{eq:P1}) always has $x = x_0 + e$ as a solution, hence is feasible even when $\sigma = 0$.

Let us briefly review the eigenvector method for interferometric recovery. In \cite{Singer}, Singer proposed to use the first eigenvector of the (noisy) data-weighted graph Laplacian as an estimator of the vector of phases. A similar idea appears in the work of Montanari et al. as the first step of their OptSpace algorithm \cite{OptSpace}, and in the work of Chatterjee on universal thresholding \cite{Chatterjee}. In our setting, this means defining 
\[ 
( \tilde{\mathcal{L}} \, )_{ij} = \left\{ \begin{array}{ll}
         d_i & \mbox{if $i = j$};\\
        - B_{ij} & \mbox{if $(i,j) \in E$};\\
        0 & \mbox{otherwise},\end{array} \right.
\]
and letting $x = \tilde{v}_1 \sqrt{n}$ where $v_1$ is the unit-norm eigenvector of $\tilde{\mathcal{L}}$ with smallest eigenvalue. Denote by $\tilde{\lambda}_1 \leq \tilde{\lambda}_2 \leq  \ldots$ the eigenvalues of $\tilde{\mathcal{L}}$. The following result is known from \cite{Bandeira}, but we provide an elementary proof (in the appendix) for completeness of the exposition.

\begin{theorem}\label{teo:eigenvector}
Assume $\| \eps \| \leq \tilde{\lambda}_2 / 2$. Then the result $x$ of the eigenvector method obeys
\[
\| x - e^{i\alpha} x_0 \| \leq \sqrt{2 n} \, \frac{\| \eps \|}{\tilde{\lambda}_2},
\]
for some $\alpha \in [0,2 \pi)$.
\end{theorem}
Alternatively, we may express the inequality in terms of $\lambda_2$, the spectral gap of the noise-free Laplacian $L$ defined earlier, by noticing\footnote{This owes to $\| \mathcal{L} - \tilde{\mathcal{L}} \| \leq \| \eps \|$, with $\mathcal{L} = \Lambda L \Lambda^*$ the noise-free Laplacian with phases introduced at the beginning of section 2.1.} that $\tilde{\lambda}_2 \geq \lambda_2 - \| \eps \|$. Both $\lambda_2$ and $\tilde{\lambda}_2$ are computationally accessible. In the case when $|B_{ij}| = 1$, we have $\tilde{\lambda}_1 \geq 0$, hence $\tilde{\lambda}_2$ is (slightly) greater than the spectral gap $\tilde{\lambda}_2 - \tilde{\lambda}_1$ of $\tilde{\mathcal{L}}$. Note that the $1/\tilde{\lambda}_2$ scaling appears to be sharp in view of the numerical experiments reported in section \ref{sec:num}. The inverse square root scaling of theorem \ref{teo:main1} is stronger in the presence of small spectral gaps, but the noise scaling is weaker in theorem \ref{teo:main1} than in theorem \ref{teo:eigenvector}.

\subsection{Interferometric recovery}

The more general version of the interferometric recovery problem is to consider a left-invertible tall matrix $A$, linear measurements $b = Ax_0$ for some vector $x_0$ (without condition on the modulus of either $b_i$ or $(x_0)_i$), noisy interferometric measurements $B_{ij} = b_i \conj{b_j} + \eps_{ij}$ for $(i,j)$ in some set $E$, and find $x$ subject to
\beq\label{eq:P3}
\sum_{(i,j) \in E \cup D} |(Ax)_i \conj{(Ax)_j} - B_{ij}| \leq \sigma.
\eeq
Notice that we now take the union of the diagonal $D = \{ (i,i) \}$ with $E$. Without loss of generality we assume that $\eps_{ij} = \conj{\eps_{ji}}$, which can be achieved by symmetrizing the measurements, i.e., substituting $\frac{B_{ij} + B_{ji}}{2}$ for $B_{ij}$.

Since we no longer have a unit-modulus condition, the relevant notion of graph Laplacian is now data-dependent. It reads
\[ 
\left( L_{|b|} \right)_{ij} = \left\{ \begin{array}{ll}
         \sum_{k:(i,k) \in E} |b_k|^2 & \mbox{if $i = j$};\\
        - |b_i| |b_j| & \mbox{if $(i,j) \in E$};\\
        0 & \mbox{otherwise}.\end{array} \right. 
\]
The connectedness properties of the underlying graph now depend on the size of $|b_i|$: the edge $(i,j)$ carries valuable information if and only if both $|b_i|$ and $|b_j|$ are large. 

A few different recovery formulations arise naturally in the context of lifting and semidefinite relaxation. 

\bit
\item The basic lifted formulation is to find some $X$ such that 
\begin{align}\label{eq:P4}
&\sum_{(i,j) \in E \cup D} | (A X A^*)_{ij} -  B_{ij} | \leq \sigma, \qquad X \succeq 0, \notag \\
&\mbox{then let $x = x_1 \sqrt{\eta_1}$, where $(\eta_1, x_1)$ is the} \notag \\ 
&\mbox{top eigen-pair of $X$.}
\end{align}
Our main result is as follows, see the appendix for the proof.

\begin{theorem}\label{teo:main2}
Assume $\| \eps \|_1 + \sigma \leq \lambda_2/2$, where $\lambda_2$ is the second eigenvalue of the data-weighted graph Laplacian $L_{|b|}$. Any solution $x$ of (\ref{eq:P4}) obeys
\[
% \| x - e^{i\alpha} x_0 \| \leq 27 \, \| x_0 \| \; \kappa(A)^2 \; \sqrt{\frac{\| \eps \|_1 + \sigma}{\lambda_2}}, VJ change of constant
\frac{\| x - e^{i\alpha} x_0 \|}{\| x_0 \|} \leq 15 \; \kappa(A)^2 \; \sqrt{\frac{\| \eps \|_1 + \sigma}{\lambda_2}},
\]
for some $\alpha \in [0,2 \pi)$, and where $\kappa(A)$ is the condition number of $A$.
\end{theorem}

The quadratic dependence on $\kappa(A)$ is necessary\footnote{
The following example shows why that is the case. For any $X_0$ and invertible $A$, the solution to $AXA^* = AX_0 A^* + \eps$ is $X = X_0 + A^{+} \eps (A^*)^+$. Let $X_0 = e_1 e_1^T$, $\eps = \delta e_1  e_1^*$ for some small $\delta$, and $A^+ = I + N e_1 e_1^T$. Then $X = (1 + \delta N^2) e_1 e_1^T$, and the square root of its leading eigenvalue is $\sqrt{\eta_1} \simeq 1 + \frac{1}{2} \delta N^2$. As a result, $x$ is perturbation of $x_0$ by a quantity of magnitude $O(\delta \| A^+ \|^2)$.}. In section \ref{sec:num}, we numerically verify the inverse square root scaling in terms of $\lambda_2$. 

%The numerical experiments also indicate that the noise scaling is not in general tight -- we do not know whether this is a consequence of the choice of regularization to pick a solution in the feasibility set or not.

If the noise originates from $b + e$ rather than $bb^* + \eps$, the error bound is again improved to
\[
% \| x - e^{i\alpha} x_0 \| \leq 27 \, \| x_0 \| \; \kappa(A)^2 \; \sqrt{\frac{\sigma}{\lambda_2}} + \| A^{-1} \| \, \| e \|, VJ change of constant
\frac{\| x - e^{i\alpha} x_0 \|}{\| x_0 \|} \leq 15 \; \kappa(A)^2 \; \sqrt{\frac{\sigma}{\lambda_2}} + \kappa(A) \, \frac{\| e \|}{\| b \|},
\]
for the same reason as earlier.

\item An alternative, two-step lifting formulation is to find $x$ through $Y$ such that
\begin{align}\label{eq:P5}
&\sum_{(i,j) \in E \cup D} | Y_{ij} -  B_{ij} | \leq \sigma, \qquad Y \succeq 0, \notag \\
&\mbox{then let $x = A^+ y_1 \sqrt{\eta_1}$, where $(\eta_1, y_1)$ is the } \notag \\ 
&\mbox{top eigen-pair of $Y$.}
\end{align}
The dependence on the condition number of $A$ is more favorable than for the basic lifting formulation.

\begin{theorem}\label{teo:main3}
Assume $\| \eps \|_1 + \sigma \leq \lambda_2/2$, where $\lambda_2$ is the second eigenvalue of the data-weighted graph Laplacian $L_{|b|}$. Any solution $x$ of (\ref{eq:P3}) or (\ref{eq:P5}) obeys
\[
% \| x - e^{i\alpha} x_0 \| \leq 27 \, \| x_0 \| \; \kappa(A) \; \sqrt{\frac{\| \eps \|_1 + \sigma}{\lambda_2}}, VJ change of constant
\frac{\| x - e^{i\alpha} x_0 \|}{\| x_0 \|} \leq 15 \; \kappa(A) \; \sqrt{\frac{\| \eps \|_1 + \sigma}{\lambda_2}},
\]
for some $\alpha \in [0,2 \pi)$.
\end{theorem} 

However, this formulation may be more computationally expensive than the one-step variant if data ($b$) space is much larger than model ($x$) space.

%\item Comment on the generalization of Singer's eigenvector method?

\eit

The quantity $\lambda_2$ is not computationally accessible in general, but it can be related to the second eigenvalue $\tilde{\lambda}_2$ of the noisy data-weighted Laplacian,
\[ 
\left( \tilde{\mathcal{L}}_{B} \right)_{ij} = \left\{ \begin{array}{ll}
         \sum_{k:(i,k) \in E} B_{kk} & \mbox{if $i = j$};\\
        - B_{ij} & \mbox{if $(i,j) \in E$};\\
        0 & \mbox{otherwise}.\end{array} \right. 
\]
It is straightforward to show that $\lambda_2 \geq \tilde{\lambda}_2 - \left[ \, (d+1) \| \eps \|_{\infty} + \| \eps \| \, \right]$, where $\| \cdot \|_{\infty}$ is the elementwise maximum on $E \cup D$, $\| \cdot \|$ is the spectral norm, and $d$ is the maximum node degree.

%\section{Algorithms}
%To solve the feasibility problems, we use a Douglas-Rachford splitting.
%Problem (\ref{eq:P3}) can be recast as finding a minimizer for 
%\[
% \imath_{\{X\succeq 0\}}(X)+\imath_{\{(AXA^*)_{ij}=B_{ij},\ (i,j)\in E\cup D\}}(X)=f(X)+g(X)
%\]
%Using that the resolvent $J_f=(I+\gamma\partial f)^{-1}$ of $f$ (resp. $g$) is the projection onto the positive semi-definite cone 
%(resp. onto the hyperplane $\{X : (AXA^*)_{ij}=B_{ij},\ (i,j)\in E\cup D\}$), the simplest algorithm to solve the problem is POCS (Projection Onto Convex Sets)
%where the iterations are 
%\[
% X_{k+1}=J_f\left(J_g(X_k)\right).
%\]
%A more efficient algorithm is the (primal-dual) Douglas-Rachford splitting, where we define $R_f=2J_f-I$, $R_g=2J_g-I$ and the iterations read 
%\begin{align*}
%& Y_{k+1}=\frac 1 2\left(R_fR_g+I\right)Y_k\\
%& X_{k}=J_g(Y_k).
%\end{align*}

%The estimates are stated with respect to the spectral gaps $\lambda_2$, $\tilde{\lambda}_2$, 
% and different norms of the noise $\eps$. \\
%LD: estimates??

\subsection{Influence of the spectral gap}

In this section, we numerically confirm the tightness of the error bound for phase recovery given by theorem \ref{teo:main1} on toy examples ($n=2^7$), with respect to the spectral gap. We perform the corresponding experiment for the situation of theorem \ref{teo:eigenvector}. In order to span a wide range of spectral gaps, three types of graphs are considered:
 \begin{itemize}
  \item the cycle graph $P_n$ which is proven to be the connected graph with the smallest spectral gap\footnote{As justified by the decreasing property of $\lambda_2$ under edge removal, mentioned earlier.};
  \item graphs obtained by adding randomly $K$ edges to $P_n$ with $K$ ranging from 1 to 50;
  \item Erd\H{o}s-R\'{e}nyi random graphs with probability ranging from 0.03 to 0.05, conditioned on connectedness (positive specrtal gap).
 \end{itemize}
A realization of the two latter types of graphs is given in figure \ref{fig:graphs}.

%\begin{figure}[htb]
\begin{figure}[H]
\centering
\includegraphics[width=.35\textwidth]{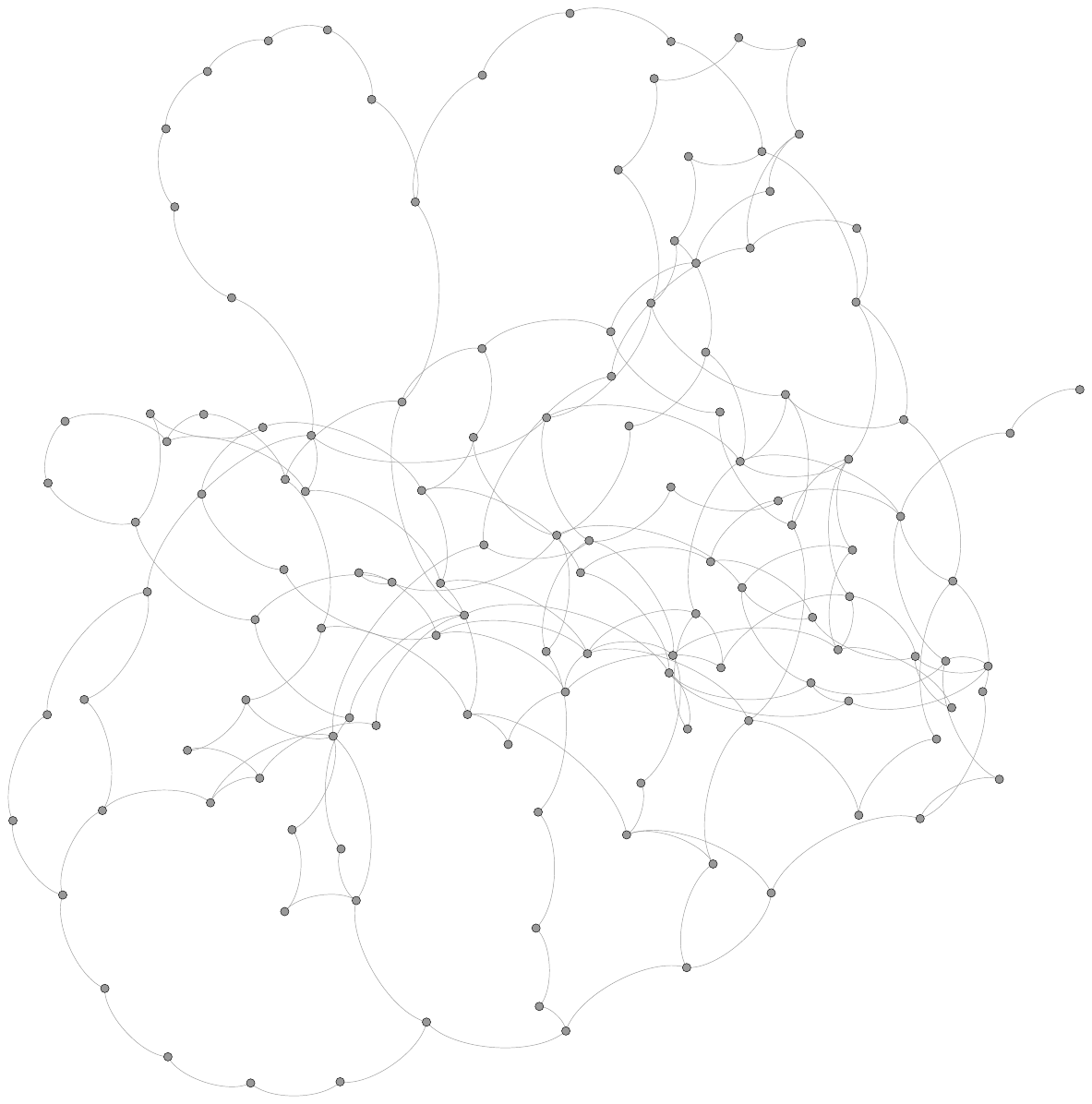} 
\includegraphics[width=.35\textwidth]{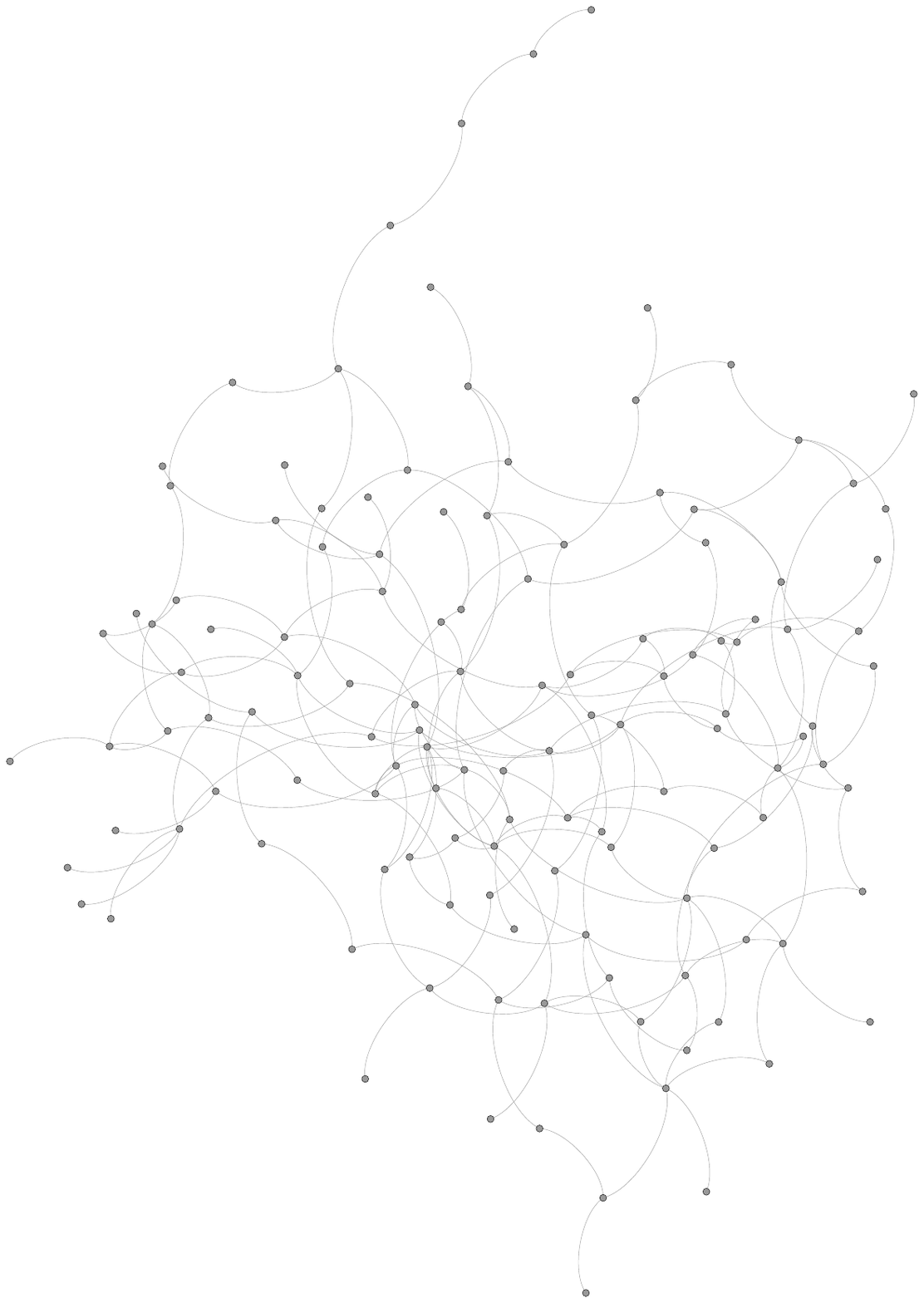}
\caption{$P_n$ + random edges (top), Erd\H{o}s-R\'{e}nyi random graph (bottom)}\label{fig:graphs}
\end{figure}

To study the eigenvector method, we draw one realization of a symmetric error matrix $\eps$ with $\eps_{ij}\sim\C\mathcal{N}(0,\eta^2)$, with $\eta=10^{-8}$. The spectral norm of the 
noise (used in theorem \ref{teo:eigenvector}) is $||\eps||\sim 2 \times 10^{-7}$.\\
For different realizations of the aforementioned graphs, we estimate the solution with the eigenvector method and plot the $\ell_2$ recovery error as a function of $\tilde{\lambda}_2$. See figure \ref{evGap}.

%\begin{figure}[htb]
\begin{figure}[H]
\includegraphics[width=\textwidth]{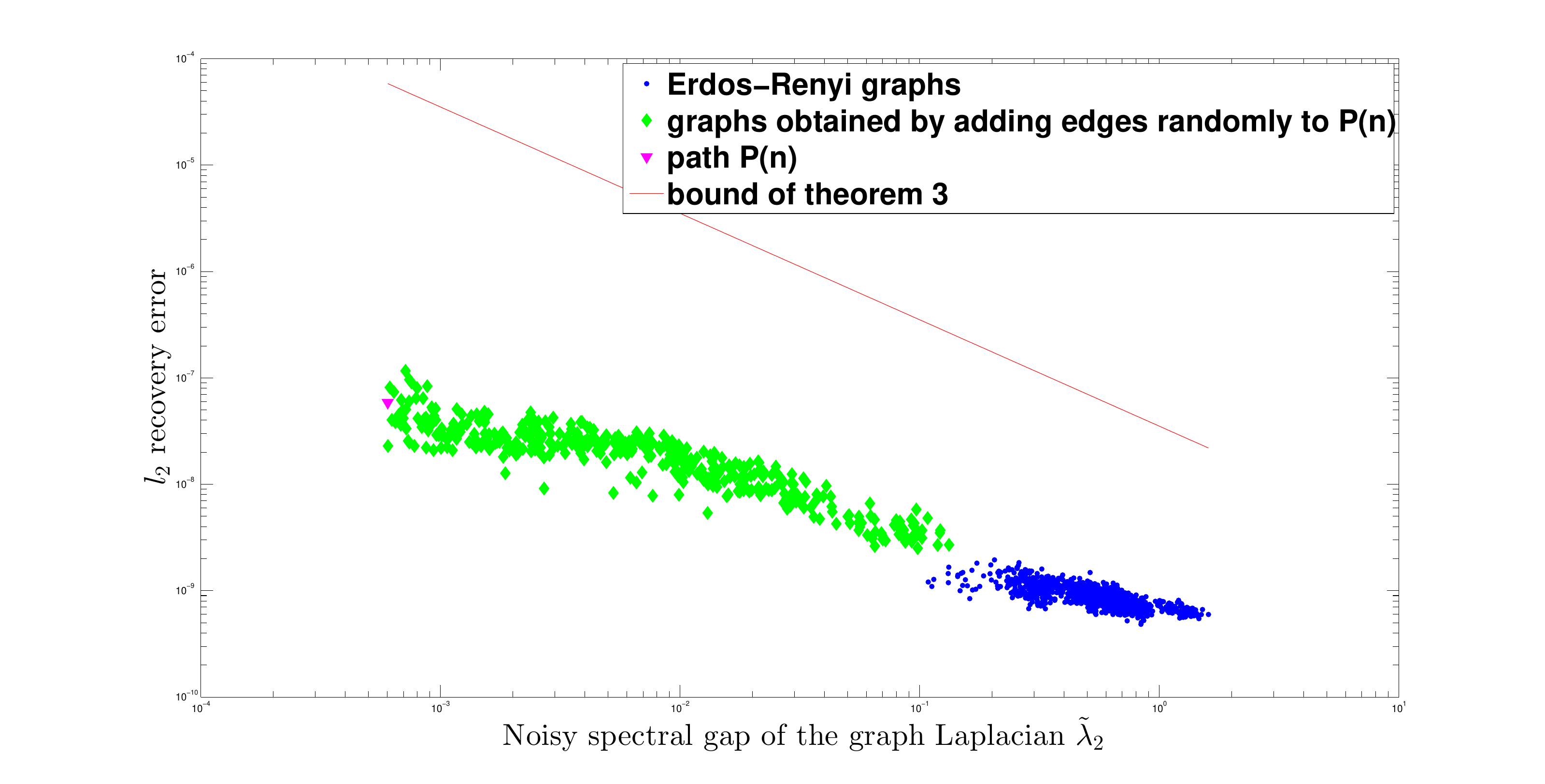}
\caption{Recovery error for the eigenvector method as a function of $\tilde{\lambda}_2$}\label{evGap}
\end{figure}

To study the feasibility method, we consider the case of an approximate fit ($\sigma=10^{-4}$) in the noiseless case ($\eps = 0$). The feasibility problem (\ref{eq:P2}) is solved using the Matlab toolbox cvx which calls the toolbox SeDuMi. An interior point algorithm (centering predictor-corrector) is used. The recovery error as a function of the spectral gap $\lambda_2$ is illustrated in figure \ref{feGap}. The square root scaling of the theorem seems to be a good match for the numerical experiments.

%\begin{figure}[htb]
\begin{figure}[H]
\includegraphics[width=\textwidth]{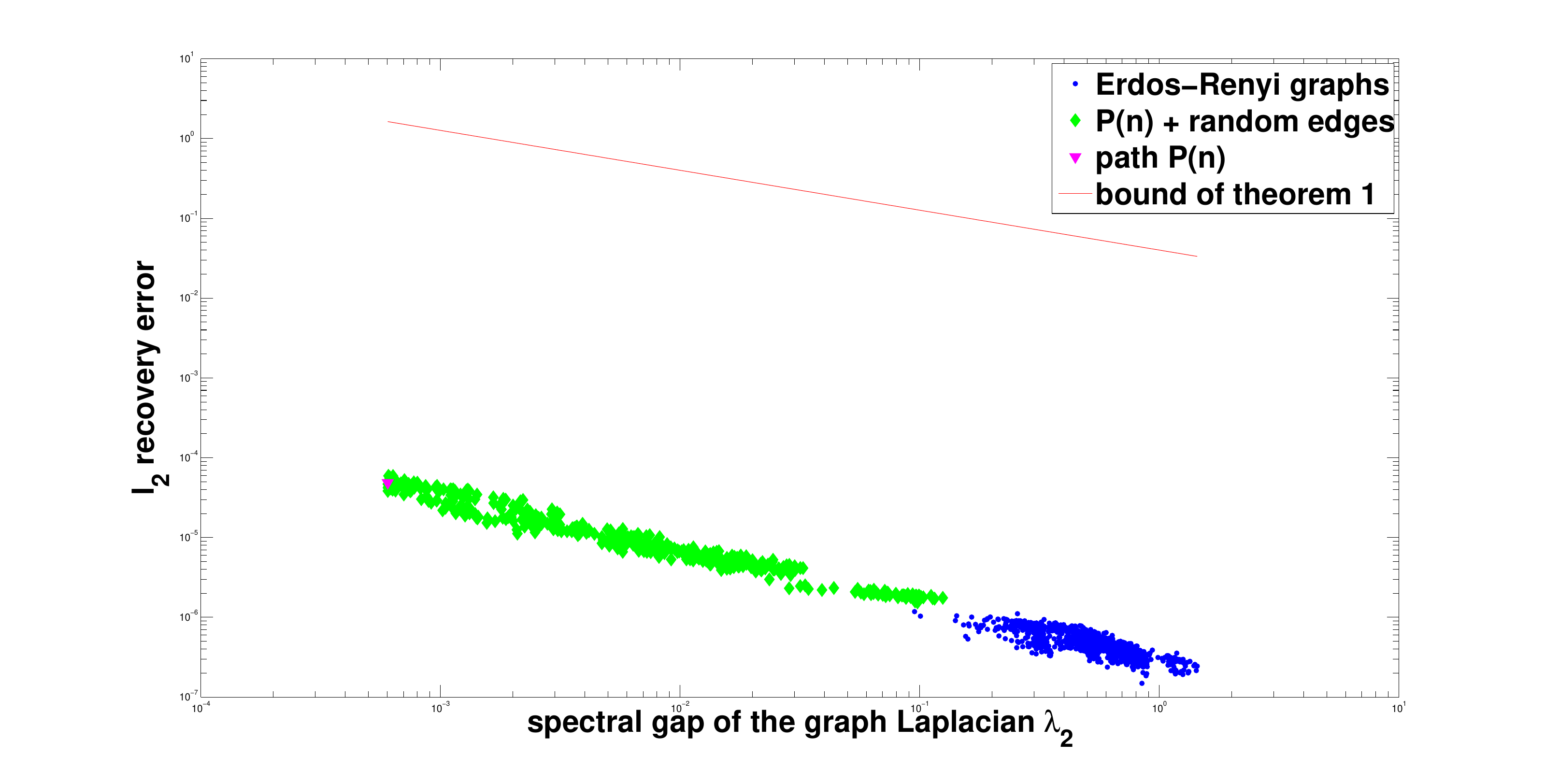}
\caption{Recovery error for the feasibility method as a function of $\lambda_2$}\label{feGap}
\end{figure}

%\subsection{Influence of the noise level}
%We now fix the set $E$ as one realization of adding $K=15$ edges randomly to $P_n$. We then draw realizations of the noise level, $\eps\sim\C\mathcal{N}(0,\eta^2)$ with $\eta$ logarithmically equally spaced between $10^{-6}\lambda_2$ and $10^{-1}\lambda_2$. The recovery for the eigenvector method is illustrated in figure \ref{evNoise}.
%
%%\begin{figure}[htb]
%\begin{figure}[H]
%\includegraphics[width=.9\textwidth]{eigen_noise}
%\caption{Recovery error for the eigenvector method as a function of the spectral norm of the noise $||\epsilon||$}\label{evNoise}
%\end{figure}
%
%For the feasibility problem, we chose $\sigma$ to be two times the $\ell_1$ norm of $\epsilon$ on $E$. The recovery for the feasibility method is illustrated in figure \ref{feNoise}. As mentioned earlier, it is unclear to us whether the bound could be strenghthened or if the discrepancy owes to the particular method by which a point is chosen in the feasibility set.
%
%
%%\begin{figure}[htb]
%\begin{figure}[H]
%\includegraphics[width=.9\textwidth]{feasibility_noise}
%\caption{Recovery error for the feasiibility method as a function of the $l_1$ norm of the noise $||\epsilon||_1$.}\label{feNoise}
%\end{figure}

\section{Numerical results: interferometric inverse scattering}\label{sec:num}

\subsection{Setup and context}

An important application of interferometric inversion is seismic imaging, where the (simplest) question is to recover of a medium's wave velocity map from recordings of waves scattered by that medium. Linear inverse problems often arise in this context. For instance, inverse source problems arise when locating microseismic events, and linear inverse scattering in the Born regime yield model updates for subsurface imaging. These linear problems all take the form
\begin{equation}
Fm=d,
\label{linear}
\end{equation}
where $F$ is the forward or modeling operator, describing the wave propagation and the acquisition, $m$ is the reflectivity model, and $d$ are the observed data. (We make this choice of notation from now on; it is standard and more handy than $Ax=b$ for what follows.) The classical approach is to use the data (seismograms) directly, to produce an image either 
\begin{itemize}
\item  by migrating the data (Reverse-time migration, RTM),
\begin{equation*}
I_{\text{RTM}}=\tilde{F}^*d
\end{equation*}
where $\tilde{F}$ is the \textbf{simulation} forward operator and $^*$ stands for the adjoint ;
\item or by finding a model that best fits the data, in a least-squares sense (Least-squares migration, LSRTM),
\begin{equation}
\label{LSM}
 m_{\text{LSM}}=\arg\min_m\frac{1}{2}||\tilde{F}m-d||_2^2.
\end{equation}
\end{itemize}
It is important to note that the physical forward operator $F$ and the one used in simulation $\tilde{F}$ can be different due to modeling errors or uncertainty. Such errors can happen at different levels:
\begin{enumerate}
\item background velocity,
\item sources and receivers positions,
\item sources time profiles.
\end{enumerate}
This list is non-exhaustive and these modeling errors have a very different effect from additive Gaussian white noise, in the sense that they induce a coherent perturbation in the data. As a result, the classical approaches (RTM, LSM) may fail in the presence of such uncertainties. 

The idea of using interferometry (i.e. products of pairs of data) to make migration robust to modeling uncertainties has already been proposed in the literature \cite{Borcea}, \cite{sava2008interferometric}, \cite{SchusterYu}, producing remarkable results. 
In their 2005 paper \cite{Borcea}, Borcea et al. developed a comprehensive framework for interferometric migration, in which they proposed the Coherent INTerfermetic imaging functional (CINT), which can be recast in our notation as
\begin{equation*}
I_{\text{CINT}}=\text{diag}\{\tilde{F}^*(E\circ [dd^*])\tilde{F}\},
\end{equation*}
where $dd^*$ is the matrix of all data pairs products, $E$ is a selector, that is, a sparse matrix with ones for the considered pairs and zeros elsewhere, $\circ$ is the entrywise (Hadamard) product, and diag is the operation of extracting the diagonal of a matrix. The CINT functional involves $\tilde{F}^*$, $\tilde{F}$ and $F$, $F^*$ (implicitly through $dd^*$), so cancellation of model errors can still occur, even when $F$ and $\tilde{F}$ are different. Through a careful analysis of wave propagation in the particular case where the uncertainty consists of random fluctuations of the background velocity, Borcea et al. derive conditions on $E$ under which CINT will be robust. 

The previous sections of this paper do not inform the robustness mentioned above, but they explain \emph{how} the power of interferometry can be extended to inversion. In a noiseless setting, the proposal is to perform inversion using selected data pair products,
\begin{equation}
\text{find }m\text{ s.t. } E\circ(\tilde{F} mm^* \tilde{F}^*)=E\circ [dd^*],
\label{intinv}
\end{equation}
i.e., we look for a model $m$ that explains the data pair products $d_i \bar{d_j}$ selected by $(i,j)\in E$. (The notation $E$ has not changed from the previous sections.) Here, $i$ and $j$ are meta-indices in data space. For example, in an inverse source problem in the frequency domain, $i \equiv (r_i,\omega_i)$ and $j \equiv (r_j,\omega_j)$, and for inverse Born scattering, $i \equiv (r_i,s_i,\omega_i)$ and $j \equiv (r_j,s_j,\omega_j)$.

A straightforward and noise-aware version of this idea is to fit products in a least-squares sense,
\begin{equation}
\label{LS}
\hat{m}_{\text{LS,pairs}}=\arg\min_m||E\circ(\tilde{F}mm^*\tilde{F}^*-dd^*)||_F^2.
\end{equation}
While the problem in (\ref{intinv}) is quadratic, the least-squares cost in (\ref{LS}) is quartic and nonconvex. The introduction of local minima is a highly undesired feature, and numerical experiments show that gradient descent can indeed converge to a completely wrong local minimizer.

A particular instance of the interferometric inversion problem is inversion from cross-correlograms. In that case, $$E_{i,j}=1\ \Leftrightarrow \omega_i=\omega_j.$$ This means that $E$ considers data pairs from different sources and receivers at the same frequency, 
$$
d_i\overline{d_j}=d(r_i,s_i,\omega)\overline{d(r_j,s_j,\omega)},
$$
where the overline stands for the complex conjugation. This expression is the Fourier transform at frequency  $\omega$ of the cross-correlogram between trace $(r_i.s_i)$ and trace $(r_j,s_j)$.

The choice of selector $E$ is an important concern. As explained earlier, it should describe a connected graph for inversion to be possible. In the extreme case where $E$ is the identity matrix, the problem reduces to estimating the model from intensity-only (phaseless) measurements, which does not in general have a unique solution for the kind of $F$ we consider. The same problem plagues inversion from cross-correlograms only: $E$ does not correspond to a connected graph in that case either, and numerical inversion typically fails. On the other hand, it is undesirable to consider too many pairs, both from a computational point of view and for robustness to model errors. For instance, in the limit when $E$ is the complete graph of all pairs $(i,j)$, it is easy to see that the quartic cost function reduces to the square of the least-squares cost function. Hence, there is a trade-off between robustness to uncertainties and quantity of information available to ensure invertibility.  It is important to stress that theorem \ref{teo:main2} gives sufficient conditions on E for recovery to be possible and stable to additive noise $\eps$, but not to modeling error (in the theorem, $F=\tilde{F}$). It does not provide an explanation of the robust behavior of interferometric inversion; see however the numerical experiments.

In the previous section, we showed that lifting convexifies problems such as (\ref{intinv}) and (\ref{LS}) in a useful way. In the context of wave-based imaging, this idea was first proposed in \cite{Chai} for intensity-only measurements. In this section's notations, we replace the optimization variable $m$ by the symmetric matrix $M=mm^*$, for which the data match becomes a linear constraint. Incorporating the knowledge we have on the solution, the problem becomes
\begin{equation*}
\begin{array}{cc}
\text{find }\quad M\quad\text{ s. t. }\\
E\circ[\tilde{F}M\tilde{F}^*]=E\circ[dd^*],\\
M\succeq 0,\\
\text{rank}(M)=1.
\end{array}
\end{equation*}
The first two constraints (data fit and positive semi-definiteness) are convex, but the rank constraint is not and would in principle lead to a combinatorially hard problem. However, as the theoretical results of this paper make clear, the rank constraint can often be dropped. We also relax the data pairs fit -- an exact fit is ill-advised because of noise and modeling errors -- to obtain the following feasibility problem equivalent to (\ref{eq:P4}),
\begin{equation}
\label{feasi}
\begin{array}{cc}
\text{find }\quad M\quad\text{ s. t. },\\
||\tilde{F}M\tilde{F}^*-dd^*||_{\ell_1(E)}\leq\sigma,\\
M\succeq 0.\\
\end{array}
\end{equation}
The approximate fit is expressed in an entry-wise $\ell_1$ sense. This feasibility problem is a convex program, for which there exist simple converging iterative methods. Once $M$ is solved for, we have already seen that a model estimate can be obtained by extracting the leading eigenvector of $M$.  \\

\subsection{A practical algorithm}

The convex formulation in (\ref{feasi}) is too costly to solve at the scale of even toy problems. Let $N$ be the total number of degrees of freedom of the unknown model $m$ ; then the variable $M$ of (\ref{feasi}) is a $N\times N$ matrix, on which we want to impose positive semi-definiteness and approximate fit. To our knowledge, there is no time-efficient and memory-efficient algorithm to solve this type of semi-definite program when $N$ ranges from $10^4$ to $10^6$.

We consider instead a non-convex relaxation of the feasibility problem (\ref{feasi}), in which we limit the numerical rank of $M$ to $K$, as in \cite{burer2003}. We may then write $M=RR^*$ where $R$ is $N\times K$ and $K\ll N$. We replace the approximate $\ell_1$ fit by Frobenius minimization. Regularization is also added to handle noise and uncertainty, yielding
\begin{equation}
\label{prac}
\hat{R}=\arg\min_R||E\circ(\tilde{F}RR^*\tilde{F}^*-dd^*)||_F^2+\lambda||R||_F^2.
\end{equation}
An estimate of $m$ is obtained from $\hat{R}$ by extracting the leading eigenvector of $\hat{R}\hat{R}^*$. Note that the Frobenius regularization on $R$ is equivalent to a trace regularization on $M=RR^*$, which is known to promote the low-rank character of the matrix. The rank-$K$ relaxation (\ref{prac}) can be seen as a generalization of the straightforward least-squares formulation (\ref{LS}). The two formulations coincide in the limit case $K=1$. The strength of (\ref{prac}) is that the optimization variable is in a slightly bigger space than formulation (\ref{LS}).

The rank-$K$ relaxation (\ref{prac}) is still nonconvex, but in practice, no local minimum has been observed even for $K=2$, whereas the issue often arises for the least-squares approach (\ref{LS}). It should also be noted that the success of inversion is simple to test \emph{a posteriori}, by comparing the sizes of the largest and second largest eigenvalues of $M$.

In practice, the memory requirement of rank-$K$ relaxation is simply $K$ times than that of non-lifted least-squares. The runtime per iteration is typically lower than $K$ times that of a least-squares step however, because the bottleneck in the computation of $\tilde{F}$ and $\tilde{F}^*$ is usually the LU factorization of a Helmholtz operator, which needs to be done only once per iteration. Empirically, the number of iterations needed for convergence does not vary much as a function of $K$. 

\subsection{Examples}

Our examples fall into three categories.
\begin{enumerate}
\item {\bf Linear inverse source problem.} Here we consider a set of receiver locations $x_r$, and a constant density acoustics inverse source problem, which reads in the Fourier domain
$$
-(\Delta+\omega^2m_0(x))\hat{u}_{s}(x,\omega)=\hat{w}(\omega)m(x)
$$
$$
Fm = d(x_r,\omega)=\hat{u}_{s}(x_r,\omega)
$$
$$
m_0(x)=\frac{1}{c_0(x)^ 2} \qquad \mbox{(squared slowness)}
$$
Waves are propagating from a source term with known time signature $w$. The problem is to reconstruct the spatial distribution $m$. \\

\item {\bf Linearized inverse scattering problem.} Here we consider a set of receivers $x_r$, waves generated by sources $x_s$, and a constant density acoustic inverse problem for the reflectivity perturbation $m_1$,
$$
\begin{array}{l}
-(\Delta+\omega^2m_0(x))\hat{u}_{0,s}(x,\omega)=\hat{w}(\omega)\delta(x-x_s)\\
-(\Delta+\omega^2m_0(x))\hat{u}_{1,s}(x,\omega)=\omega^2\hat{u}_{0,s}(x,\omega)m_1(x)
\end{array}
$$
$$
Fm_1=d(x_r,x_s,\omega)=\hat{u}_{1,s}(x_r,\omega).
$$
The isolation of the Born scattered wavefield (primary reflections) $u_{1,s}$ from the full scattered field, although a very difficult task in practice, is assumed to be performed perfectly in this paper. 

\item {\bf Full waveform inversion (FWI).} We again consider a set of receivers $x_r$, waves generated by sources $x_s$, and a constant density acoustic inverse problem for the reflectivity $m$, with 
$$
-(\Delta+\omega^2 m(x))\hat{u}_s(x,\omega)=\hat{w}(\omega)\delta(x-x_s)\\
$$
$$
\mathcal{F}(m)=d(x_r,x_s,\omega)=\hat{u}_s(x_r,\omega).
$$

\end{enumerate}

In a first numerical experiment (Figure \ref{fig:seismic1}), we consider a linearized inverse scattering problem where $c = 1/\sqrt{m}$ is the Marmousi2 p-wave velocity model \cite{Marmousi2}, $m_0$ is a smoothed version of $m$, and $m_1 = m - m_0$ is the model perturbation used to generated data in the linearized forward model. The sources and receivers are placed at the top of the image in an equispaced fashion, with 30 sources and 300 receivers. The frequency sampling is uniform between 3 and 10 Hz, with 10 frequency samples. The Helmholtz equation is discretized with fourth-order finite differences at about 20 points per wavelength for data modeling, while the simulations for the inversion are with second-order finite differences. The noise model is taken to be Gaussian,
\[
\tilde{d}_i=d_i+\eta_i\quad \eta_i\sim\mathbb{C}\mathcal{N}(0,\sigma^2)
\]
where $\displaystyle\sigma=0.1\frac{||b||_2}{\sqrt{2n}}$, so that $\displaystyle\frac{||\eta||_2}{||b||_2}=0.1$ (10\% additive noise). Figure \ref{fig:seismic1} shows stable recovery of $m_1$ from noisy $d$, both by least-squares and by interferometric inversion. In this case the graph $E$ is taken to be an Erd\H{o}s-R\'{e}nyi random graph with $\displaystyle p=1.5\frac{log(N)}{N}$ to ensure connectedness. (Instabilities occur if $p$ is smaller; larger values of $p$ do not substantially help.) Note that if $E$ were chosen as a disconnected graph that only forms cross-correlations (same $\omega$ for the data indices $i$ and $j$), then interferometric inversion is not well-posed and does not result in good images (not shown). The optimization method for interferometric inversion is the rank-2 relaxation scheme mentioned earlier. The message of this numerical example is twofold: it shows that stable recovery is possible in the interferometric regime, when $E$ is properly chosen; and a comparison of Figures \ref{fig:seismic1} bottom and middle shows that it does not come with the loss of resolution that would be expected from CINT \cite{Borcea}. This observation was also a key conclusion in the recent paper by Mason et al. \cite{Yazici}.

In a second numerical experiment (Figure \ref{fig:seismic2}), we consider full waveform inversion where $c = 1/\sqrt{m}$ is the Marmousi2 p-wave velocity model, and the initial model is the same $m_0$ as used in the previous example. The setup for the acquisition, the Helmholtz equation, the noise, and the mask $E$ are the same as before. Figure \ref{fig:seismic2} shows the result of quasi-Newton (LBFGS) iterations with a frequency sweep, and the corresponding interferometric inversion result. To produce this latter result, we simply replace the classical adjoint-state gradient step by interferometric inversion applied to the data residual. In all the numerical experiments so far (Figures \ref{fig:seismic1} and \ref{fig:seismic2}), the values of the model misfits are not meaningfully different in the least-squares and interferometric cases. The message of this numerical example is that there seems to be no loss of resolution in the full waveform case either, when comparing Figure \ref{fig:seismic2} (bottom) to its least-squares counterpart (second from bottom).

Possible limitations of the interferometric approach are the slightly higher computational and memory cost (as discussed in the previous section); the need to properly choose the data mask $E$; and the fact that the data match is now quartic rather than quadratic in the original data $d$. Quartic objective functions can be problematic in the presence of outliers, such as when the noise is heavy-tailed, because they increase the relative importance of those corrupted measurements. We did not attempt to deal with heavy-tailed noise in this paper.

It is also worth noting that ``backprojecting the cross-correlations" (or the more general quadratic combinations we consider here) can be shown to be related to the first iteration of gradient descent in the \emph{lifted semidefinite formulation} of interferometric inversion.

So far, our numerical experiments merely confirm the theoretical prediction that interferometric inversion can be accurate and stable under minimal assumptions on the mask/graph $E$ of data pair products. In the next section, we show that there is also an important \emph{rationale} for switching to the interferometric formulation: its results display robustness vis-a-vis some uncertainties in the forward model $F$.

%\begin{figure}[htb]
\begin{figure}[htb]
\includegraphics[width=\textwidth, height=4cm]{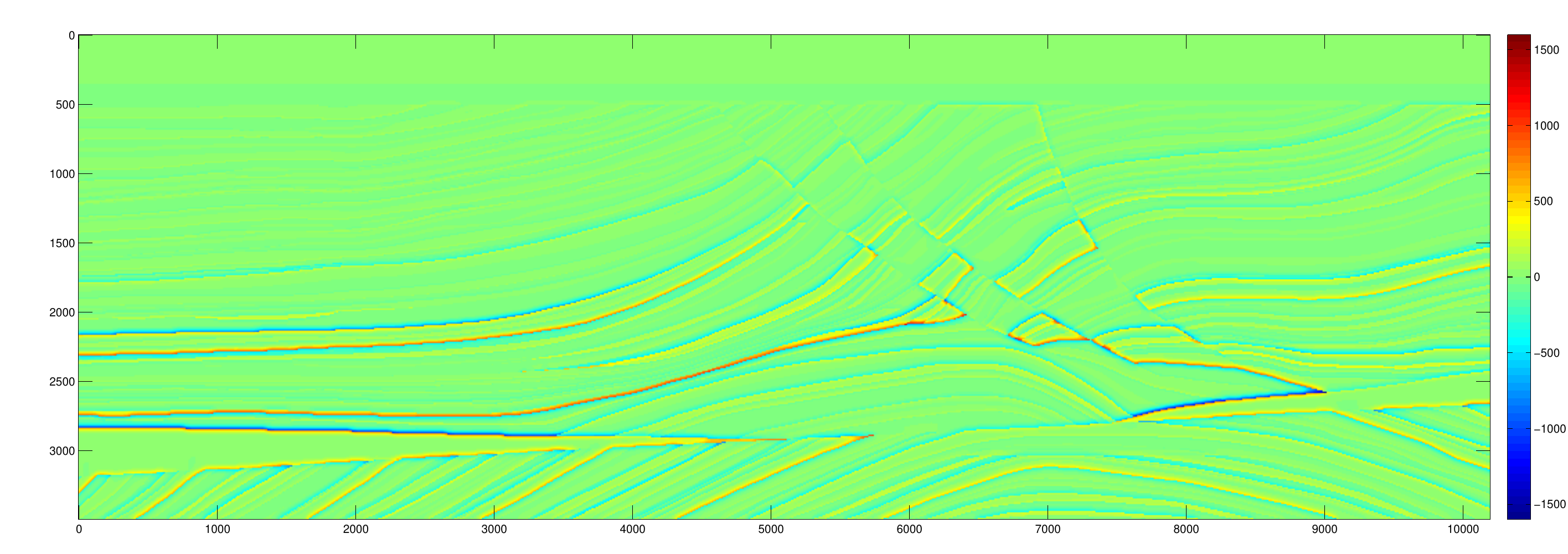} \\
\includegraphics[width=\textwidth, height=4cm]{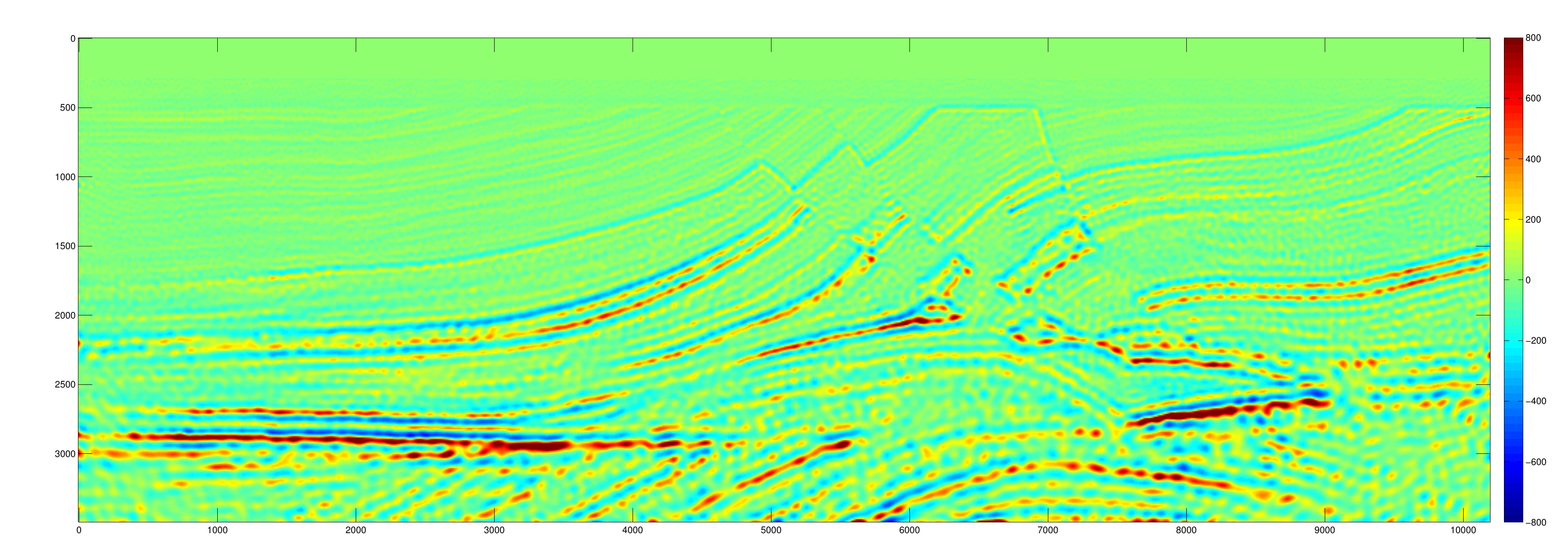} \\
\includegraphics[width=\textwidth, height=4cm]{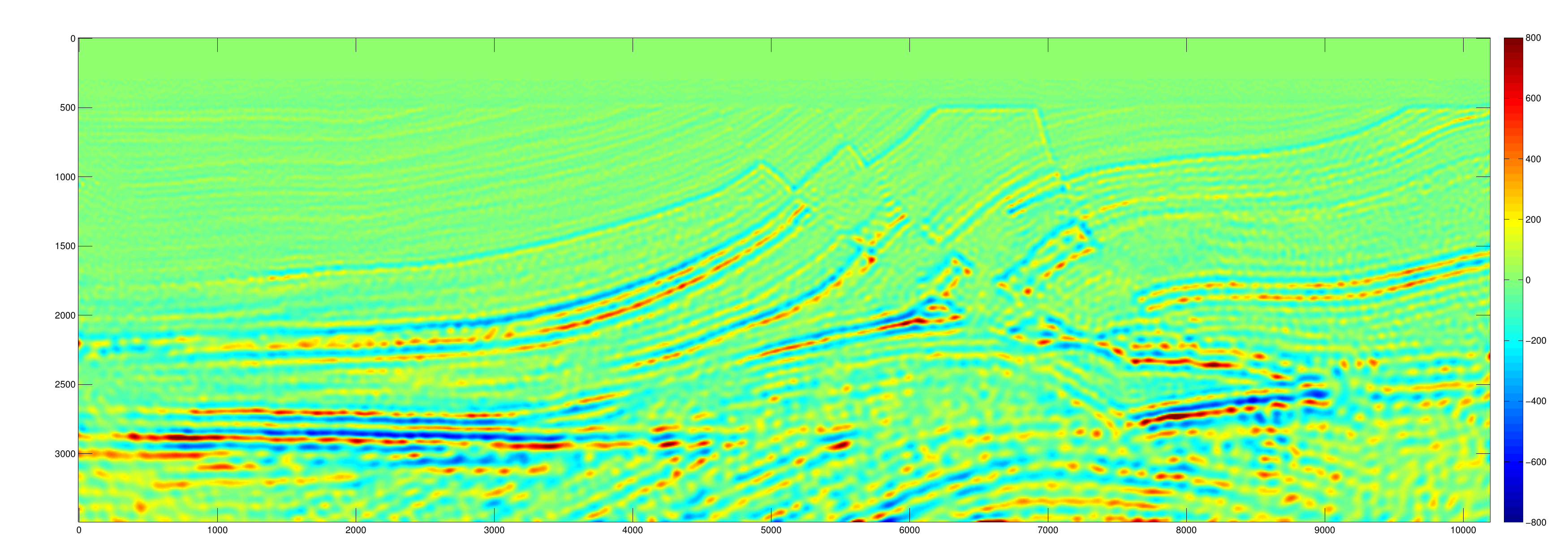}
\caption{Top: true, unknown reflectivity profile $m_1$ used to generate data. Middle: least-squares solution. Bottom: result of interferometric inversion. This example shows that interferometric inversion is stable when the mask $E$ is connected, as in the theory, and shows no apparent loss of resolution vs. least squares.}\label{fig:seismic1}
\end{figure}

%\begin{figure}[htb]
\begin{figure}[htb]
\includegraphics[width=\textwidth, height=4cm]{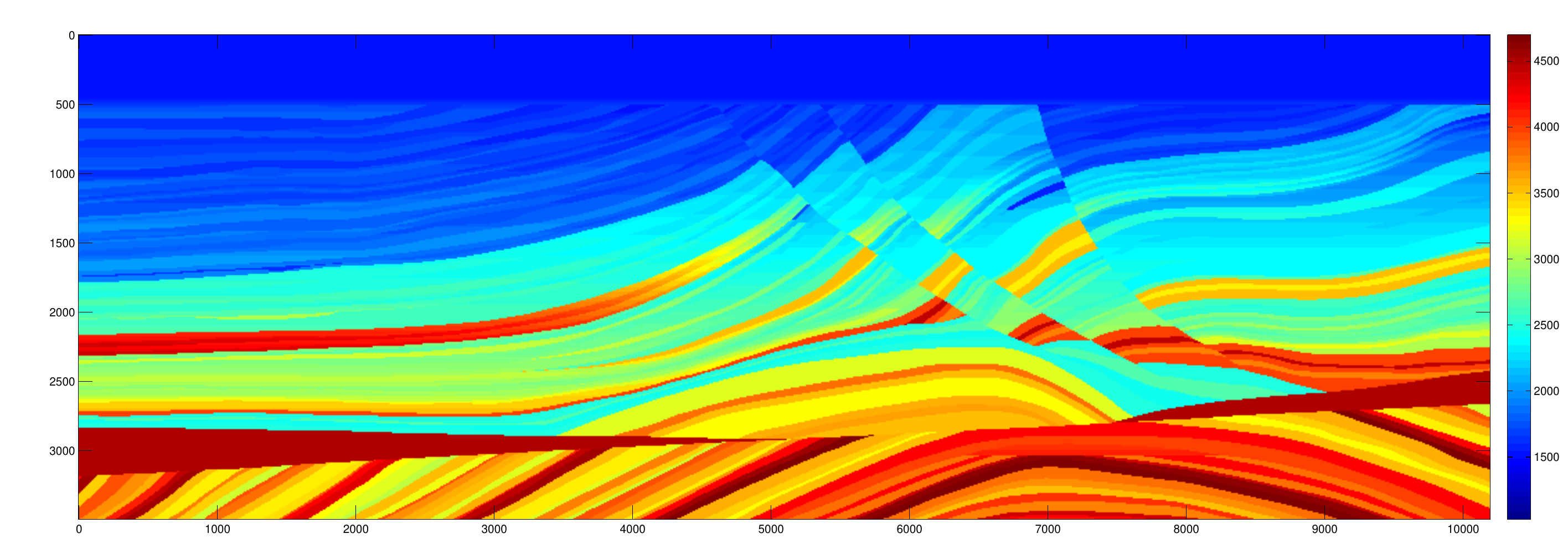} \\
\includegraphics[width=\textwidth, height=4cm]{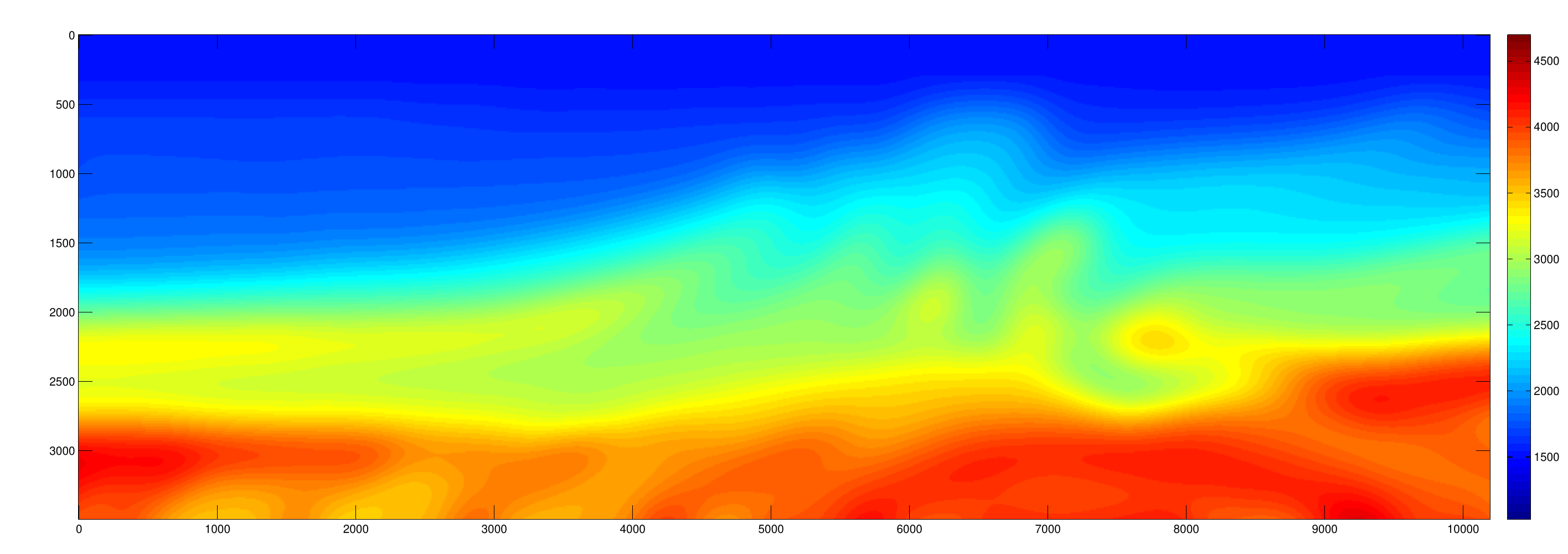} \\
\includegraphics[width=\textwidth, height=4cm]{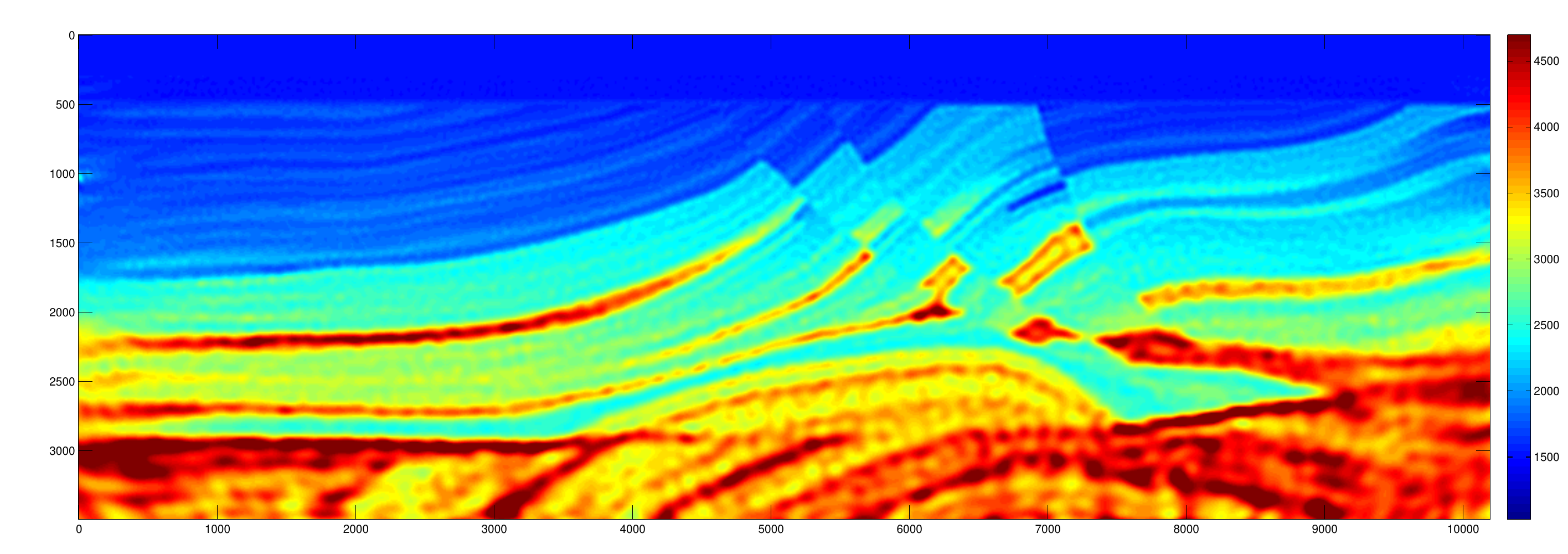} \\
\includegraphics[width=\textwidth, height=4cm]{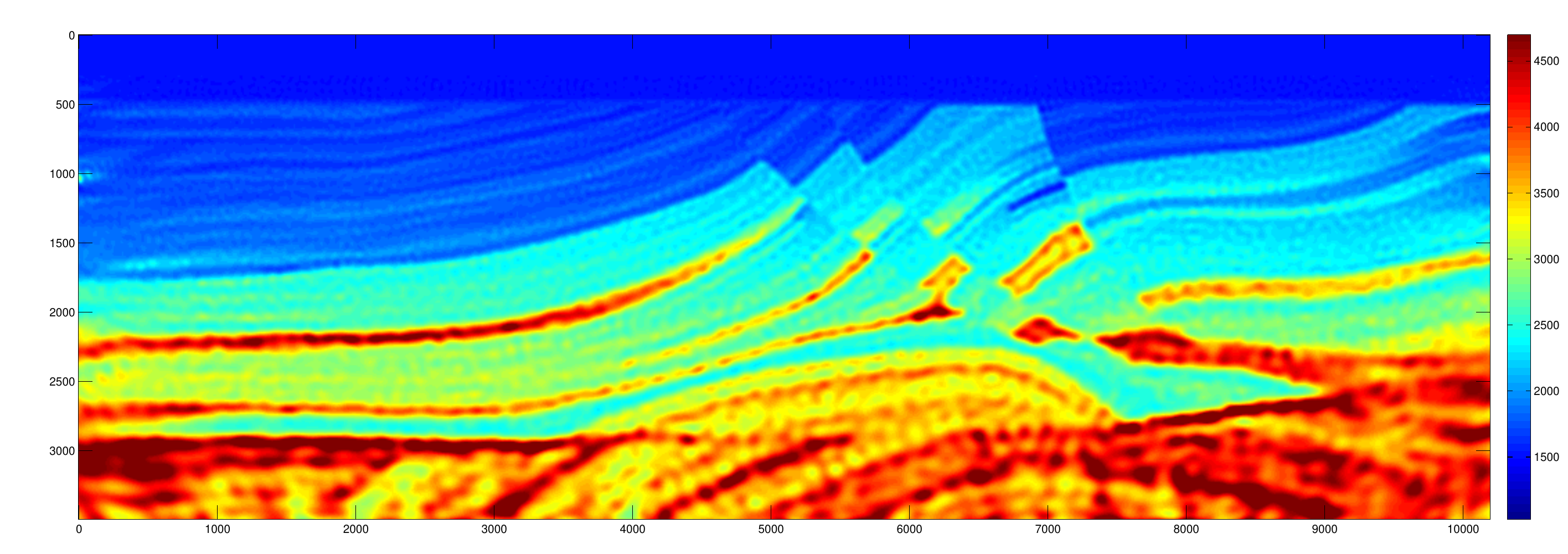}
\caption{Top: true, unknown map of the squared slowness $m$ used to generate data. Second: initial guess for either inversion scheme. Third: nonlinear least-squares solution, a.k.a. full waveform inversion. Bottom: result of interferometric inversion. This example shows that the stability and resolution properties of interferometric inversion carry over to the case of full waveform inversion.}\label{fig:seismic2}
\end{figure}

\subsection{Model robustness}

In this section we continue to consider wave-based imaging scenarios, but more constrained in the sense that the receivers and/or sources completely surround the object to be imaged, i.e., provide a full aperture. The robustness claims below are only for this case. On the other hand, we now relax the requirement that the forward model $\tilde{F}$ used for inversion be identical or very close to the forward map $F$ used for data modeling. The interferometric mask $E$ is also different, and more physically meaningful than in the previous section: it indicates a small band around the diagonal $i=j$, i.e., it selects nearby receiver positions and frequencies, as in \cite{Borcea}. The parameters of this selector were tuned for best results; they physically correspond to the idea that two data points $d_i$ and $d_j$ should only be compared if they are within one cycle of one another for the most oscillatory phase in $d$. In all our experiments with model errors, it is crucial that the data misfit parameter $\sigma$ be positive nonzero -- it would be a mistake to try and explain the data perfectly with a wrong model.

In the next numerical experiment (Figure \ref{fig:robust_isp}), we consider the inverse source problem, where the source distribution $m$ is the Shepp-Logan phantom. This community model is of interest in ultrasound medical imaging, where it represents a horizontal cross-section of different organs in the torso. The receivers densely surround the phantom and are depicted as white crosses. Equispaced frequencies are considered on the bandwidth of $w$. A small modeling error is assumed to have been made on the background velocity: in the experiment, the waves propagated with unit speed $c_0(x) = 1$, but in the simulation, the waves propagate more slowly, $\tilde{c}_0(x) = 0.95$. As shown in Figure \ref{fig:robust_isp} (middle),  least-squares inversion does not properly handle this type of uncertainty and produces a defocused image. (Least-squares inversion is regularized with the $\ell_2$ norm of the model, a.k.a. Tykhonov. A large range of values of the regularization parameter was tested, and the best empirical results are reported here.) In contrast, interferometric inversion, shown in Figure \ref{fig:robust_isp} (bottom), enjoys a better resolution.  The price to pay for focusing is positioning: although we do not have a mathematical proof of this fact, the interferometric reconstruction is near a shrunk version of the true source distribution. 

In the last numerical experiment (Figure \ref{fig:robust_born}), we consider the linearized inverse scattering problem, where the phantom is now the reflectivity perturbation $m_1$. As the figure shows, sources and receivers surround the phantom, with a denser sampling for the receivers than for the sources. The wave speed profile is uniform ($c=1$). In this example, the modeling error is assumed to be on the receiver positions: they have a smooth random deviation from a circle, as shown in Figure \ref{fig:robust_born} (top). Again, least-squares inversion produces a poor result regardless of the Tykhonov regularization parameter (Figure \ref{fig:robust_born}, middle), where the features of the phantom are not clearly reconstructed, and the strong outer layer is affected by a significant oscillatory error. Interferometric inversion produces a somewhat more focused reconstruction (Figure \ref{fig:robust_born}, bottom), where more features of the phantom are recovered and the outer layer is well-resolved.

Model robustness is heuristically plausible in the scenario when data are of the form $d_i \sim e^{i \omega \tau_i}$ for some traveltimes $\tau_i$ which are themselves function of a velocity $c$ through $\tau_i = \delta_i/c$. In that case, the combination $d_i \overline{d_j} \sim e^{i \omega (\tau_i - \tau_j)}$ has a phase that encodes the idea of a \emph{traveltime difference}. When $i$ is near $j$ in data space (because they correspond, say, to nearby receivers), it is clear that $\tau_i - \tau_j$ depends in a milder fashion on errors in $c$, or on correlated errors in $\delta_i$, than the individual traveltimes themselves. This results in some degree of stability of selected $d_i \overline{d_j}$ with respect to those types of errors, which in turn results in more focused imaging.  This phenomenon is not unlike the celebrated statistical stability of $d_i \overline{d_j}$ under some models of randomness on either the velocity or the sources. For random media, this behavior was leveraged in the context of coherent interferometric imaging (CINT) in the 2011 work of Borcea et al. \cite{CINT2011}

\begin{figure}[htb]
\begin{center}
\vspace{-.5cm}
 \includegraphics[width=.49\textwidth]{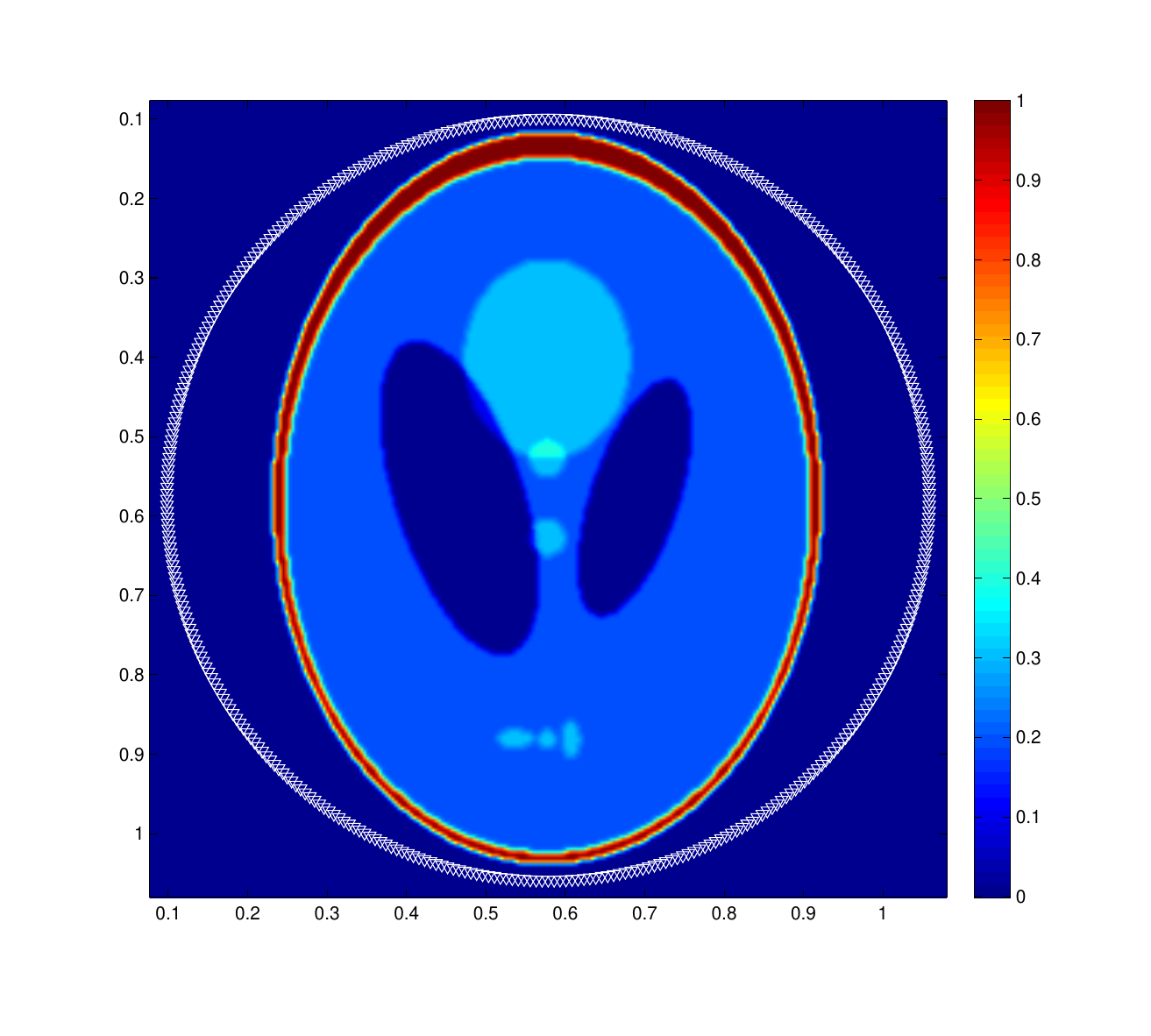} \\
 \vspace{-1cm}
 \includegraphics[width=.49\textwidth]{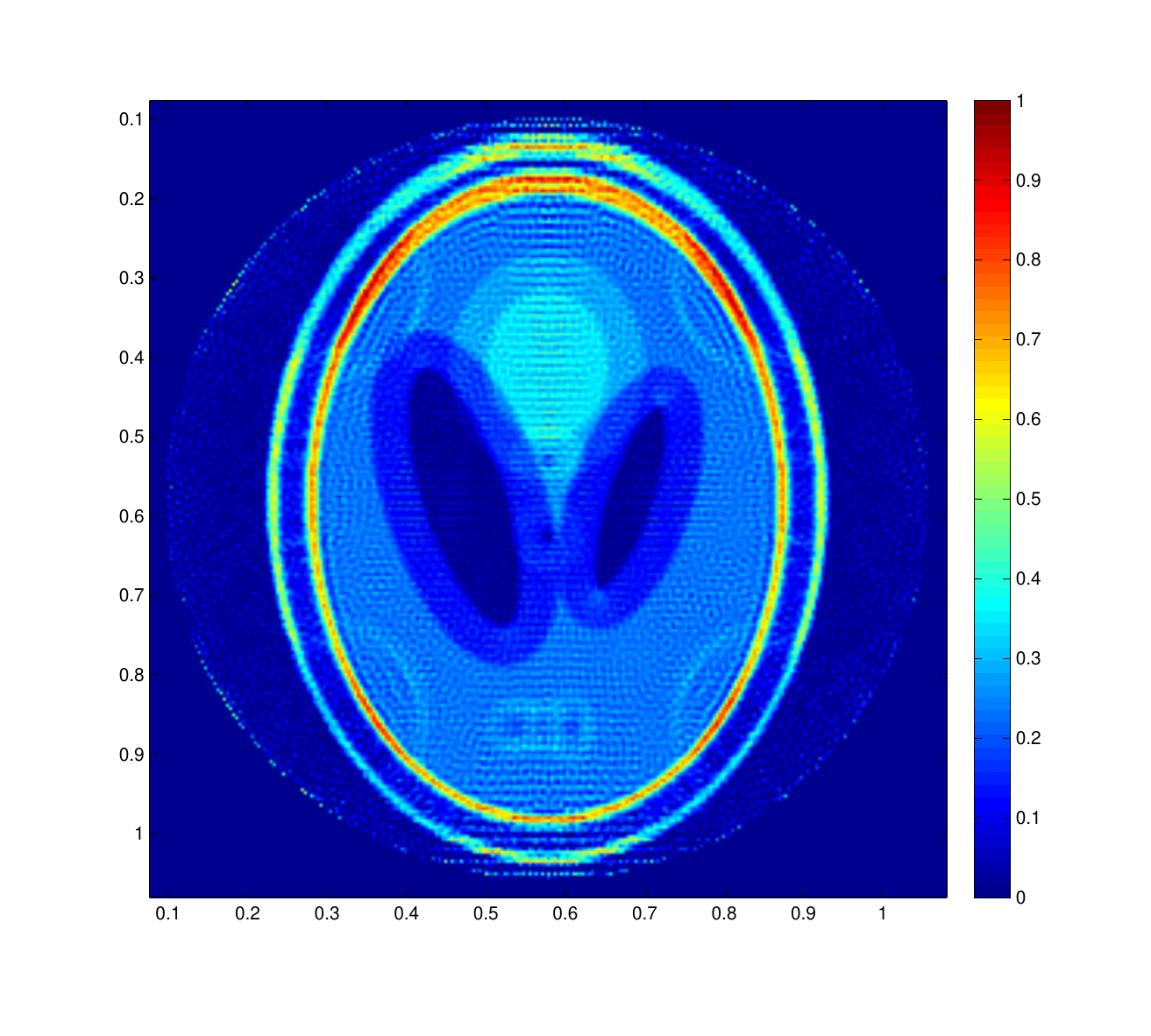} \\
 \vspace{-1cm}
 \includegraphics[width=.49\textwidth]{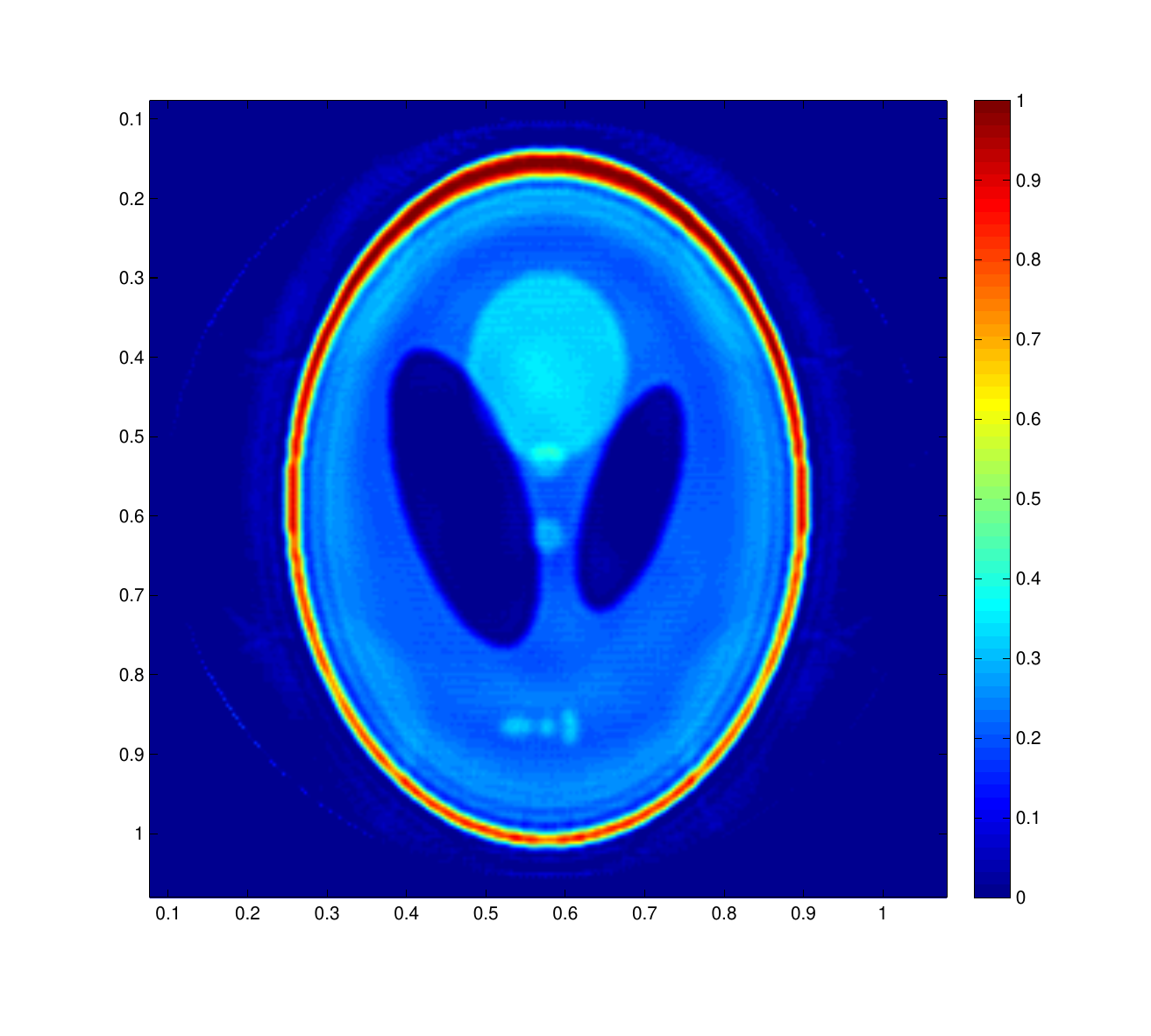} \\
 \vspace{-1cm}
  \caption{Top: Setting of the inverse source experiment.  Middle: least-squares reconstruction. Bottom: interferometric reconstruction. This example shows model robustness for focused imaging, in the specific case of full aperture, when the underlying constant wave speed is slightly wrong in the forward model.}
   \label{fig:robust_isp}
 \end{center}
\end{figure} 

\begin{figure}[htb]
\begin{center}
\vspace{-.5cm}
 \includegraphics[width=.49\textwidth]{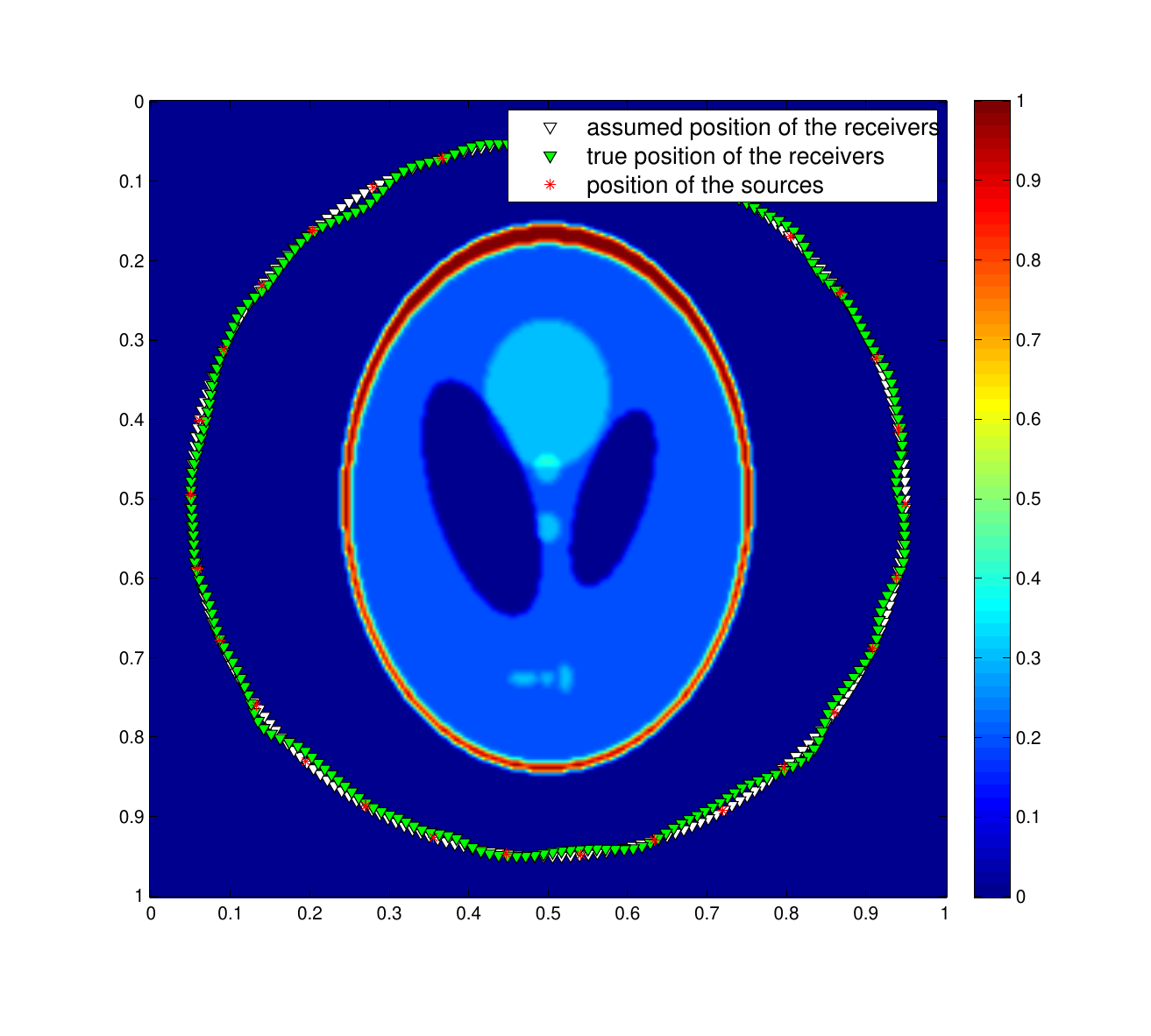} \\
 \vspace{-1cm}
 \includegraphics[width=.49\textwidth]{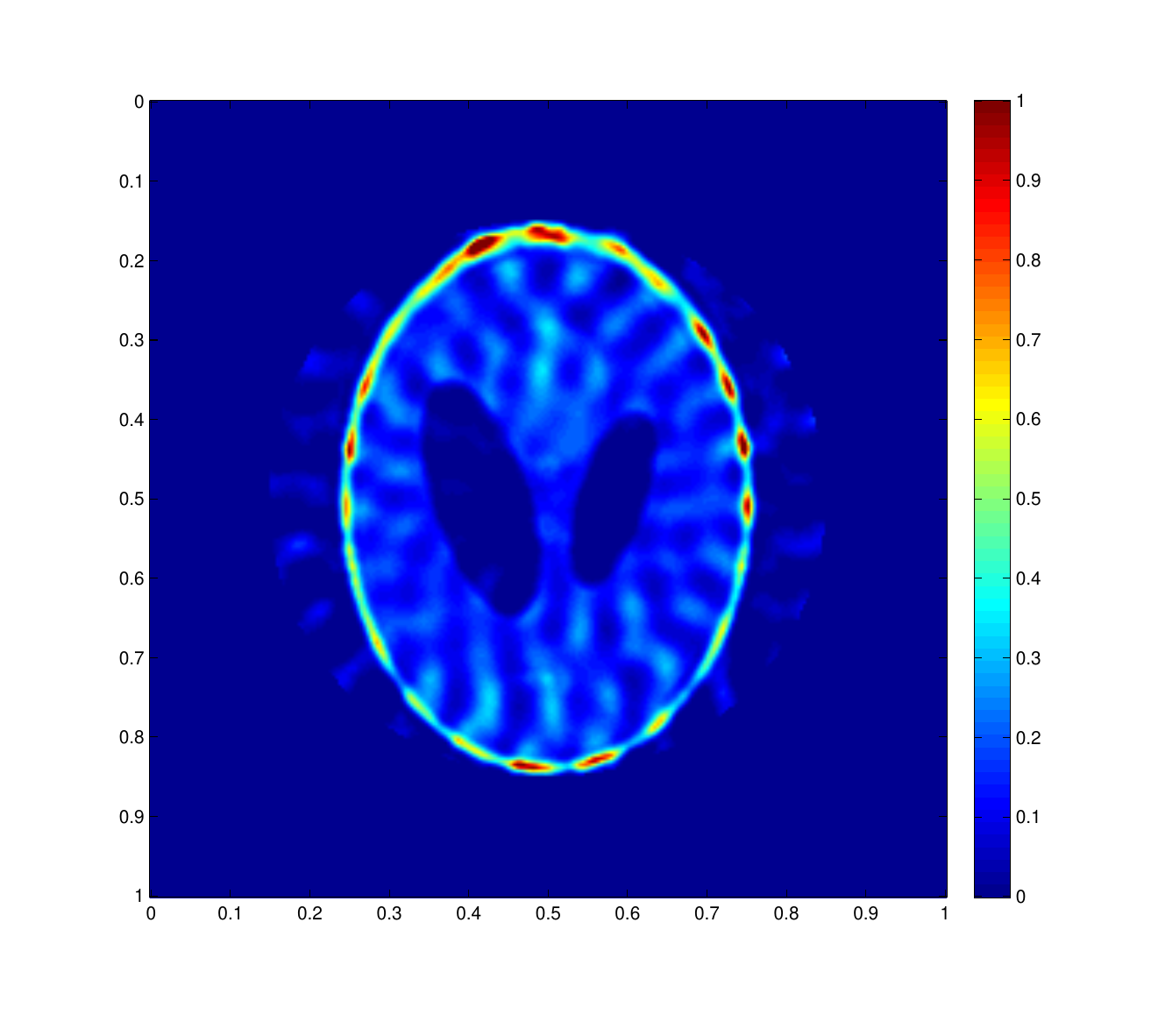} \\
 \vspace{-1cm}
 \includegraphics[width=.49\textwidth]{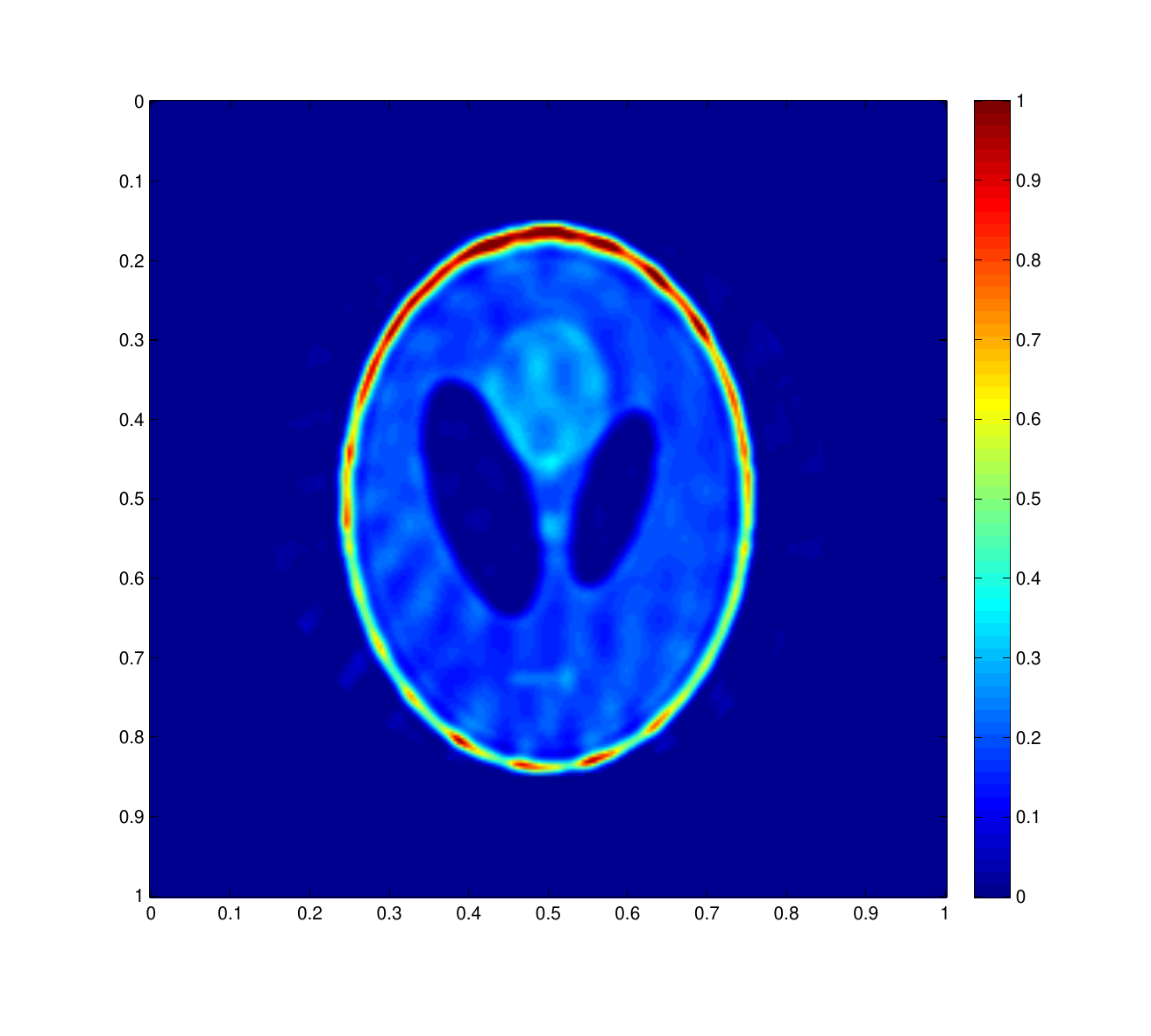} \\
 \vspace{-1cm}
  \caption{Top: Setting of the inverse scattering experiment.  Middle: least-squares reconstruction. Bottom: interferometric reconstruction. This example shows model robustness for focused imaging, in the specific case of full aperture, when the receiver locations are slightly wrong in the forward model.}
   \label{fig:robust_born}
 \end{center}
\end{figure} 

\section{Discussion}

This paper explores a recent optimization idea, convex relaxation by lifting, for dealing with the quadratic data combinations that occur in the context of interferometry. The role of the mask $E$ that selects the quadratic measurements is explained: it is shown that it should encode a well-connected graph for robustness of the recovery. The recovery theorems do not assume a random model on this graph; instead, they involve the Laplacian's spectral gap. 

Numerical experiments in the context of imaging are shown to confirm that recovery is stable when the assumptions of the theorem are obeyed. Developing methods to solve the lifted problem at interesting scales is a difficult problem that we circumvent via an ad-hoc rank-2 relaxation scheme. It is observed empirically that there is no noticeable loss of resolution from switching to an interferometric optimization formulation. It is also observed that interferometric inversion for imaging can display a puzzling degree of robustness to very specific model uncertainties. 

We leave numerous important questions unanswered, such as how to optimize the choice of the selector $E$ when there is a trade-off between connectivity and robustness to model uncertainties; how to select the data misfit parameter $\sigma > 0$ and whether there is a phase transition in $\sigma$ for model robustness vs. lack thereof; how to justify the gain in focusing in some reconstructions; and whether this gain in focusing is part of a tradeoff with the geometric faithfulness of the recovered image.

\appendix
\section{Proofs}

\subsection{Proof of theorem \ref{teo:main1}.}\label{sec:proof1}

Observe that if $x$ is feasible for (\ref{eq:P1}), then $xx^*$ is feasible for (\ref{eq:P2}), and has $e^{i \alpha} x$ as leading eigenvector. Hence we focus without loss of generality on (\ref{eq:P2}).

As in \cite{Singer}, consider the Laplacian matrix weighted with the unknown phases,
\[
\mathcal{L} = \Lambda L \Lambda^*,
\]
with $\Lambda = \mbox{diag}(x_0)$. In other words $\mathcal{L}_{ij} = (X_0)_{ij} L_{ij}$ with $X_0 = x_0 x_0^*$. We still have $\mathcal{L} \succeq 0$ and $\lambda_1 = 0$, but now $v_1 = \frac{1}{\sqrt{n}} x_0$. Here and below, $\lambda$ and $v$ refer to $\mathcal{L}$, and $v$ has unit $\ell_2$ norm.

The idea of the proof is to compare $X$ with the rank-1 spectral projectors $v_j v_j^*$ of $\mathcal{L}$.  Let $\< A, B \> = \tr(A B^*)$ be the Frobenius inner product. Any $X$ obeying (\ref{eq:P1}) can be written as $X = X_0 + \tilde{\eps}$ with $\| \tilde{\eps} \|_1 \leq \| \eps \|_1 + \sigma$. We have
\begin{align*}
\< X, \mathcal{L} \> = \< X_0, \mathcal{L} \> + \< \tilde{\eps}, \mathcal{L} \>
\end{align*}
A short computation shows that
\begin{align*}
\< X_0, \mathcal{L} \> &= \sum_i (X_0)_{ii} \conj{\mathcal{L}_{ii}} + \sum_{(i,j) \in E} (X_0)_{ij}  \conj{\mathcal{L}_{ij}} \\
&= - \sum_i d_i + \sum_{(i,j) \in E} (x_0)_i \conj{(x_0)_j} \; \conj{(x_0)_i} (x_0)_j \\
&= \sum_i \left[ - d_i + \sum_{j: (i,j) \in E} 1\right] \\
&= 0.
\end{align*}
Since $|L_{ij}| = 1$ on $E$, the error term is simply bounded as
\[
| \< \tilde{\eps}, \mathcal{L} \> | \leq \| \tilde{\eps} \|_1
\]
%(where the $\ell_1$ norm is componentwise, not induced.)

On the other hand the Laplacian expands as
\[
\mathcal{L} = \sum_{j} v_j \lambda_j  v_j^*,
\]
so we can introduce a convenient normalization factor $1/n$ and write 
\beq\label{eq:convcomb}
\< \frac{X}{n}, \mathcal{L} \> = \sum_{j} c_j \lambda_j, 
\eeq
with
\[
c_j = \< \frac{X}{n}, v_j v_j^* \> = \frac{ v_j^* X v_j}{n}.
\]
Notice that $c_j \geq 0$ since we require $X \succeq 0$. Their sum is
\[
\sum_j c_j = \< \frac{X}{n}, \sum_j v_j v_j^* \> = \< \frac{X}{n}, I \> = \frac{\tr(X)}{n} = 1.
\]
Hence (\ref{eq:convcomb}) is a convex combination of the eigenvalues of $\mathcal{L}$, bounded by $\| \tilde{\eps} \|_1 / n$. The smaller this bound, the more lopsided the convex combination toward $\lambda_1$, i.e., the larger $c_1$. The following lemma makes this observation precise.

\begin{lemma}\label{teo:one}
Let $\mu = \sum_j c_j \lambda_j$ with $c_j \geq 0$, $\sum_j c_j = 1$, and $\lambda_1 = 0$. If $\mu \leq \lambda_2$, then
\[
c_1 \geq 1 - \frac{\mu}{\lambda_2}.
\]
\end{lemma}
\begin{proof}[Proof of lemma \ref{teo:one}.]
\[
\mu = \sum_{i \geq 2} c_j \lambda_j \geq \lambda_2 \sum_{j \geq 2} c_j = \lambda_2 (1-c_1),
\]
then isolate $c_1$.
\end{proof}

Assuming $\| \tilde{\eps} \|_1 \leq n \lambda_2$, we now have
\[
\< \frac{X}{n}, v_1 v_1^* \> \geq 1 - \frac{\| \tilde{\eps} \|_1}{n \lambda_2}.
\]
We can further bound
\begin{align*}
\| \frac{X}{n} - v_1 v_1^* \|_F^2 &= \tr \left[ \left( \frac{X}{n} - v_1 v_1^* \right)^2 \right] \\
% &= \tr ( v_1 v_1^*) + \frac{\tr(X)}{n^2} - 2 \; \tr \left( \frac{X}{n} v_1 v_1^* \right). 
&= \tr ( (v_1 v_1^*)^2) + \frac{\tr(X^2)}{n^2} - 2 \; \tr \left( \frac{X}{n} v_1 v_1^* \right). %VJ : corrected typos
\end{align*}
The first term is 1. The second term is less than 1, since $\tr(X^2)\leq\tr(X)^2$ % $\tr(X)^2 \leq \tr(X^2)$ VJ the inequality is the other way around 
for positive semidefinite matrices. Therefore,
\[
\| \frac{X}{n} - v_1 v_1^* \|_F^2 \leq 2 - 2 \; \tr \left( \frac{X}{n} v_1 v_1^* \right) \leq \frac{2 \| \tilde{\eps} \|_1}{n \lambda_2}.
\]

We can now control the departure of the top eigenvector of $X/n$ from $v_1$ by the following lemma. It is analogous to the sin theta theorem of Davis-Kahan, except for the choice of normalization of the vectors.  (It is also a generalization of a lemma used by one of us in \cite{DemanetHand} (section 4.2).) The proof is only given for completeness.

\begin{lemma}\label{teo:sintheta}
Consider any Hermitian $X \in \C^{n \times n}$, and any $v \in \C^n$, such that 
$\| X - v v^* \| < \frac{\| v \|^2}{2}$. Let $\eta_1$ be the leading eigenvalue of $X$, and $x_1$ the corresponding unit-norm eigenvector. Let $x$ be defined either as (a) $x_1 \| v \|$, or as (b) $x_1 \sqrt{\eta_1}$. Then
\[
\| \; x \| x \| - e^{i \alpha} v \| v \| \; \| \leq 2 \sqrt{2} \; \| X - v v^* \|,
\]
for some $\alpha \in [0,2 \pi)$.

\end{lemma}
\begin{proof}[Proof of Lemma \ref{teo:sintheta}]
Let $\delta = \| X - v v^* \|$. Notice that $\| vv^* \| = \| v \|^2$. Decompose $X = \sum_{j=1}^n x_j \eta_j x_j^*$ with eigenvalues $\eta_j$ sorted in decreasing order. By perturbation theory for symmetric matrices (Weyl's inequalities),
\begin{equation}\label{eq:pert_eig}
\max \{ | \| v \|^2 - \eta_1| , |\eta_2|, \ldots, |\eta_n| \} \leq \delta,
\end{equation}
so it is clear that $\eta_1 > 0$, and that the eigenspace of $\eta_1$ is one-dimensional, as soon as $\delta < \frac{\| v \|^2}{2}$. 

Let us deal first with the case (a) when $x = x_1 \| v \|$. Consider
\[
v v^* - x x^* = v v^* - X + Y,
\]
where
\[
Y = x_1 (\| v \|^2 - \eta_1) x_1^* + \sum_{j=2}^n x_j \eta_j x_j^*.
\]
From (\ref{eq:pert_eig}), it is clear that $\| Y \| \leq \delta$. Let $v_1 = v / \| v \|$. We get
\[
\| v v^* - x x^* \| \leq \| v v^* - X \| + \| Y \| \leq 2 \delta.
\]
Pick $\alpha$ so that $| v^* x | = e^{- i \alpha} v^* x$. Then
\begin{align*}
&\| \; v \| v \|  - e^{- i \alpha} x \| x \| \; \|^2 \\
&= \| v \|^4 + \| x \|^4 - 2 \, \| v \| \, \| x \| \, \Re \,  e^{- i \alpha}  v^* x \\
&= \| v \|^4 + \| x \|^4 - 2 \, \| v \| \, \| x \| \, | v^* x | \qquad \mbox{by def. of $\alpha$} \\
&\leq \| v \|^4 + \| x \|^4 - 2 \, | v^* x |^2 \qquad \mbox{by Cauchy-Schwarz} \\
&= \| vv^* - xx^* \|_F^2 \\
&\leq 2 \| vv^* - xx^* \|^2 \qquad \mbox{since $ vv^* - xx^*$ has rank 2} \\
&\leq 8 \delta^2.
\end{align*}

The case (b) when $x = x_1 \sqrt{\eta_1}$ is treated analogously. The only difference is now that
\[
Y = \sum_{j=2}^n x_j \eta_j x_j^*.
\]
A fortiori, $\| Y \| \leq \delta$ as well.

\end{proof}

Part $(a)$ of lemma \ref{teo:sintheta} is applied with $X/n$ in place of $X$, and $v_1$ in place of $v$. In that case, $\| v_1 \| = 1$. We conclude the proof by noticing that $v_1 = \frac{x_0}{\sqrt{n}}$, and that the output $x$ of the lifting method is normalized so that $x_1 = \frac{x}{\sqrt{n}}$.

\subsection{Proof of theorem \ref{teo:eigenvector}}

The proof is a simple argument of perturbation of eigenvectors. We either assume $\eps_{ij} = \conj{\eps_{ji}}$ or enforce it by symmetrizing the measurements. Define $\mathcal{L}$ as previously, and notice that $\| \mathcal{L} - \tilde{\mathcal{L}} \| \leq \| \eps \|$. Consider the eigen-decompositions
\[
\mathcal{L} v_j = \lambda_j v_j, \qquad \tilde{\mathcal{L}} \tilde{v}_j = \tilde{\lambda}_j \tilde{v}_j,
\]
with $\lambda_1 = 0$. Form
\[
\tilde{\mathcal{L}} v_j = \lambda_j v_j + r_j,
\]
with $\| r_j \| \leq \| \eps \|$. Take the dot product of the equation above with $\tilde{v}_k$ to obtain
\[
\< \tilde{v}_k , r_j \> = (\tilde{\lambda}_k - \lambda_j) \< \tilde{v}_k , v_j \>.
\]
Let $j = 1$, and use $\lambda_1 = 0$. We get
\[
\sum_{k \geq 2} | \< \tilde{v}_k, v_1 \> |^2 \leq \frac{\sum_{k \geq 2} | \< \tilde{v}_k, r_1 \> |^2} {\max_{k \geq 2} | \tilde{\lambda}_k |^2} \leq \frac{\| \eps \|^2}{\tilde{\lambda}^2_2}.
\]
%We write the denominator with $\tilde{\lambda}_2$ in place of $\lambda_2$, by noticing that 
%\[
%| \, \lambda_2 - \| \eps \| \, | = | \, \tilde{\lambda}_2 - (\lambda_2 - \tilde{\lambda_2}) - \| \eps \| \, | \geq | \, \tilde{\lambda}_2 - 2 \| \eps \| \, |
%\]
As a result,
\[
| \< \tilde{v}_1, v_1 \> |^2 \geq 1 - \frac{\| \eps \|^2}{\tilde{\lambda}^2_2}.
\]
Choose $\alpha$ so that $\< e^{i \alpha} \tilde{v}_1, v_1 \> = | \< \tilde{v}_1, v_1 \> |$. Then
\begin{align*}
\| v_1 - e^{i \alpha} \tilde{v}_1 \|^2 &= 2 - 2 \Re \< e^{i \alpha} \tilde{v}_1, v_1 \> \\
&=  2 - 2 | \< \tilde{v}_1, v_1 \> | \\
&\leq 2 - 2  | \< \tilde{v}_1, v_1 \> |^2 \\
&\leq 2 \frac{\| \eps \|^2}{\tilde{\lambda}_2^2}.
\end{align*}
Conclude by multiplying through by $n$ and taking a square root.

\subsection{Proof of theorem \ref{teo:main2}.}

The proof follows the argument in section \ref{sec:proof1}; we mostly highlight the modifications.

Let $b_i = |b_i| e^{i \phi_i}$. The Laplacian with phases is $\mathcal{L}_b = \Lambda_\phi L_{|b|} \Lambda^*_\phi$, with $\Lambda_\phi = \mbox{diag}(e^{i \phi_i})$. Explicitly,
\[ 
\left( \mathcal{L}_b \right)_{ij} = \left\{ \begin{array}{ll}
         \sum_{k: (i,k) \in E} |b_k|^2 & \mbox{if $i = j$};\\
        - b_i \conj{b_j} & \mbox{if $(i,j) \in E$};\\
        0 & \mbox{otherwise},\end{array} \right. 
\]
The matrix $Y = A X A^*$ is compared to the rank-1 spectral projectors of $\mathcal{L}_b$. We can write it as $Y  =  bb^* + \tilde{\eps}$ with $\| \tilde{\eps} \|_1 \leq \| \eps \|_1 + \sigma$. The computation of $\< bb^*, \mathcal{L}_b \>$ is now
\begin{align*}
\< bb^*, \mathcal{L}_b \> &= \sum_i b_i \conj{b_i} \conj{\mathcal{L}_{ii}} + \sum_{(i,j) \in E} b_i \conj{b_j} \,\conj{\mathcal{L}_{ij}} \\
&= - \sum_i |b_i|^2 \sum_{k:(i,k) \in E} |b_k|^2 + \sum_{(i,j) \in E} b_i \conj{b_j} \; \conj{b_i} b_j \\
&= \sum_i |b_i|^2 \left[ -  \sum_{j:(i,j) \in E} |b_j|^2 + \sum_{j:(i,j) \in E} |b_j|^2 \right] \\
&= 0.
\end{align*}
The error term is now bounded (in a rather crude fashion) as
\begin{align*}
| \< \tilde{\eps}, \mathcal{L}_b \> | &\leq \| \mathcal{L}_b \|_\infty \| \tilde{\eps} \|_1 \\ 
&\leq \left[ \max_i \sum_{j:(i,j) \in E} |b_j|^2 \right] \; \| \tilde{\eps} \|_1 \leq \| b \|^2 \| \tilde{\eps} \|_1.
\end{align*}
Upon normalizing $Y$ to unit trace, we get
\[
| \< \frac{Y}{\tr(Y)}, \mathcal{L}_b \> | \leq \frac{\| b \|^2 \| \tilde{\eps} \|_1}{\| b \|^2 + \tr(\tilde{\eps})} \leq 2 \| \tilde{\eps} \|_1,
\]
where the last inequality follows from
\begin{align*}
|\tr(\tilde{\eps})| &\leq \| \tilde{\eps} \|_1 \\
&\leq \| \eps \|_1 + \sigma \\
&\leq \lambda_2 / 2 \qquad \mbox{(assumption of the theorem)} \\
&\leq \| b \|^2/2 \qquad \mbox{(by Gershgorin)}.
\end{align*}

On the other hand, we expand 
\[
\< \frac{Y}{\tr(Y)}, \mathcal{L}_b \> = \sum_{j} c_j \lambda_j,
\]
and use $X \succeq 0 \Rightarrow Y \succeq 0$ to get $c_j \geq 0, \; \sum_j c_j = 1$. Since $2 \| \tilde{\eps} \|_1 \leq \lambda_2$, we conclude as in section \ref{sec:proof1} that 
\[
\< \frac{Y}{\tr(Y)}, v_1 v_1^* \> \geq 1- \frac{2 \| \tilde{\eps} \|_1}{\lambda_2},
\]
hence
\begin{equation}\label{eq:Ybound}
\| \frac{Y}{\tr(Y)} - v_1 v_1^* \|_F^2 \leq 4 \frac{\| \tilde{\eps} \|_1}{\lambda_2}.
\end{equation}
For $X = A^{+} Y (A^*)^+$, we get
\[
\| \frac{X}{\tr(Y)} - (A^{+} v_1) (A^{+} v_1)^* \|_F^2 \leq 4 \| A^{+} \|^4 \frac{\| \tilde{\eps} \|_1}{\lambda_2}.
\]
Call the right-hand side $\delta^2$. Recall that $v_1 = b / \|b\|$ hence $A^{+} v_1 = x_0 / \|b\|$. Using $\tr(Y) = \| b \|^2 + \tr (\tilde{\eps})$, we get
\begin{equation}\label{eq:Xbound}
\| X - x_0 x_0^* \| \leq \delta \, \tr (Y) + \frac{\| x_0 \|^2}{\| b \|^2} \, | \tr (\tilde{\eps}) |.
\end{equation}
Elementary calculations based on the bound $\| \tilde{\eps} \|_1 \leq \lambda_2/2 \leq \| b \|^2 / 2$ allow to further bound the above quantity by 
% $(2 + \sqrt{2}) \delta \| b \|^2$ VJ it seems you can get the better constant below
$\frac{(6 + \sqrt{2})}{4} \delta \| b \|^2$
. We can now call upon lemma \ref{teo:sintheta}, part (b), to obtain
\[
\| x \| x \| - e^{i \alpha} x_0 \| x_0 \| \| \leq 2 
%\sqrt{2} (2 + \sqrt{2}) \delta \| b \|^2, VJ tracking the change of constant
\sqrt{2} \frac{(6 + \sqrt{2})}{4} \delta \| b \|^2,
\]
where $x = x_1 \sqrt{\lambda_1(X)}$ is the leading eigenvector of $X$ normalized so that $\| x \|^2 = \lambda_1(X)$ is the leading eigenvalue of $X$. We use (\ref{eq:Xbound}) one more time to bound
\[
% | \lambda_1(X) - \| x_0 \|^2 | \leq (2 + \sqrt{2}) \delta \| b \|^2, VJ tracking the change of constant
| \lambda_1(X) - \| x_0 \|^2 | \leq \frac{(6 + \sqrt{2})}{4} \delta \| b \|^2,
\]
hence
\begin{align*}
&\| x_0 \| \, \| x - e^{i \alpha} x_0 \| \\
&\leq \| x \| x \| - e^{i \alpha} x_0 \| x_0 \| + \| x \| \, | \| x \| - \| x_0 \| | \\
% &\leq 2 \sqrt{2} (2 + \sqrt{2}) \delta \| b \|^2 + \frac{\| x \|}{\| x \| + \| x_0 \|} | \| x \|^2 - \| x_0 \|^2 | \\
% &\leq (2 \sqrt{2} + 1) (2 + \sqrt{2}) \delta \| b \|^2. VJ tracking the change of constant
&\leq 2 \sqrt{2} \frac{(6 + \sqrt{2})}{4} \delta \| b \|^2 + \frac{\| x \|}{\| x \| + \| x_0 \|} | \| x \|^2 - \| x_0 \|^2 | \\
&\leq (2 \sqrt{2} + 1) \frac{(6 + \sqrt{2})}{4} \delta \| b \|^2.
\end{align*}
Use $\| b \| \leq \| A \| \, \| x_0 \|$ and the formula for $\delta$ to conclude that
\[
\| x - e^{i \alpha} x_0 \| \leq C \, \| x_0 \| \, \kappa(A)^2 \, \sqrt{\frac{ \| \tilde{\eps} \|_1}{\lambda_2}},
\]
% with $C = 2 (2 \sqrt{2} + 1) (2 + \sqrt{2}) \leq 27$.
with $C = 2 (2 \sqrt{2} + 1) \frac{(6 + \sqrt{2})}{4} \leq 15$.

%We normalize the rank-1 term by observing that $v_1 = b / \|b\|$, $A^{-1} v_1 = x_0 / \|b\|$, hence
%\[
%(A^{-1} v_1) (A^{-1} v_1)^* = \frac{\| x_0 \|^2}{\| b \|^2} \; \frac{x_0 x_0^*}{\| x_0 \|^2}.
%\]
%It follows that, for some scalar $c$,
%\begin{align*}
%\| c X - \frac{x_0 x_0^*}{\| x_0 \|^2} \|_F^2 &\leq 4 \frac{\| b \|^4}{\| x_0 \|^4} \| A^{-1} \|^4 \frac{\| \eps \|_1}{\lambda_2} \\
%&\leq 4 \kappa(A)^4 \frac{\| \eps \|_1}{\lambda_2}.
%\end{align*}
%We conclude by applying lemma \ref{teo:sintheta}.

\subsection{Proof of theorem \ref{teo:main3}.}

The proof proceeds as in the previous section, up to equation (\ref{eq:Ybound}). The rest of the reasoning is a close mirror of the one in the previous section, with $Y$ in place of $X$, $y$ in place of $x$, $b$ in place of $x_0$, and $\delta$ re-set to $2 \sqrt{\| \tilde{\eps} \|_1 / \lambda_2}$. We obtain
\[
\| y - e^{i \alpha} b \| \leq 15 \, \| b \| \, \sqrt{\frac{ \| \tilde{\eps} \|_1}{\lambda_2}}.
\]
We conclude by letting $x  = A^+ y$, $x_0 = A^+ b$, and using $\| b \| \leq \| A \| \| x_0 \|$.

\end{document}